\documentclass{amsart}[12pt]
\usepackage{amsmath,amssymb,pifont,calc,bm}

\newtheorem{theorem}{Theorem}[section]
\newtheorem{proposition}[theorem]{Proposition}
\newtheorem{lemma}[theorem]{Lemma}
\newtheorem{corollary}[theorem]{Corollary}

\newtheorem*{theorema}{Theorem A}

\theoremstyle{definition}
\newtheorem{definition}[theorem]{Definition}

\theoremstyle{remark}
\newtheorem{remark}[theorem]{Remark}
\newtheorem*{remarks*}{Remarks}

\numberwithin{equation}{section}

\newcommand{\CC}{{\mathbb C}}

\newcommand{\ZZ}{{\mathbb Z}}
\newcommand{\RR}{{\mathbb R}}
\newcommand{\LL}{{\mathbb L}}
\newcommand{\TT}{{\mathbb T}}
\newcommand{\KK}{{\mathbb K}}

\newcommand{\PP}{{\mathbb P}}

\newcommand{\U}{{\mathcal U}}

\newcommand{\fg}{{\mathfrak g}}
\newcommand{\fk}{{\mathfrak k}}

\DeclareMathOperator{\tr}{Tr}

\newcommand{\Log}{\mbox{\textrm Log}}

\renewcommand{\span}{\mbox{\textrm span}}

\newcommand{\res}{\mbox{\textrm res}}
\newcommand{\Vol}{\mbox{\textrm Vol}}
\newcommand{\Hom}{\mbox{\textrm Hom}}
\newcommand{\pr}{\mbox{\textrm pr}}

\begin{document}
\openup3pt

\title{Stability functions}

\author{Daniel Burns}\address{Department of Mathematics \\
University of Michigan \\ Ann Arbor \\ MI \\ 48109 \\ USA}
\thanks{Daniel Burns is supported in part by NSF grant DMS-0514070.}
\email{dburns@umich.edu}
\author{Victor Guillemin}\address{Department of Mathematics\\
Massachusetts Institute of Technology \\Cambridge, MA 02139 \\
USA}\thanks{Victor Guillemin is supported in part by NSF grant
DMS-0408993.}\email{vwg@math.mit.edu}
\author{Zuoqin Wang}
\address{Department of Mathematics\\
Massachusetts Institute of Technology \\ Cambridge, MA  02139 \\
USA} \thanks{Zuoqin Wang is supported in part by NSF grant
DMS-0408993.}\email{wangzq@math.mit.edu}

\date{}

\begin{abstract}
In this article we discuss the role of \emph{stability functions}
in geometric invariant theory and apply stability function
techniques to various types of asymptotic problems in the K\"ahler geometry of GIT quotients. 
We discuss several particular classes of examples, namely, toric varieties, spherical varieties and the symplectic version of quiver varieties.
\end{abstract}

\maketitle

%\tableofcontents

%{Oct 16, 2008}

%##############################
%        Section 1            %
%##############################

\section{Introduction}
\label{sec:1}

Suppose $(M, \omega)$ is a pre-quantizable K\"ahler manifold, and
$(\LL, \langle \cdot, \cdot \rangle)$ is a pre-quantization of
$(M, \omega)$. Let $\nabla$ be the metric connection on $\LL$.
Then the quantizability assumption is equivalent to the condition
that the curvature form equals the negative of the K\"ahler form,
   \begin{equation}
   \label{eqn:1.1}
   curv(\nabla) = -\omega.
   \end{equation}
For any positive integer $k$ we will denote the $k^{th}$ tensor
power of $\LL$ by $\LL^k$. The Hermitian structure on $\LL$
induces a Hermitian structure on $\LL^k$. Denote by
$\Gamma_{hol}(\LL^k)$ the space of holomorphic sections of
$\LL^k$. (If $M$ is compact, $\Gamma_{hol}(\LL^k)$ is a finite
dimensional space, whose dimension is given by the Riemann-Roch
theorem,
   \begin{equation}
   \label{eqn:1.2}
   \dim \Gamma_{hol}(\LL^k)  = k^d \Vol(M) + k^{d-1} \int_M
   c_1(M)\wedge \omega^{d-1} + \cdots \ .
   \end{equation}
for $k$ sufficiently large.) We equip this space with the $L^2$
norm induced by the Hermitian structure,
   \begin{equation}
   \label{eqn:1.3}
   \langle s_1, s_2 \rangle  = \int_M \langle s_1(x),
   s_2(x) \rangle \frac{\omega^d}{d!}.
   \end{equation}
(In semi-classical analysis $\Gamma_{hol}(\LL^k)$ is the Hilbert
space of quantum states associated to $M$, and $k$ plays the role
of the inverse of Planck's constant.)

If one has a holomorphic action of a compact Lie group on $M$
which lifts to $\LL$ one gets from the data above a Hermitian line
bundle on the geometric invariant theory (GIT) quotient of $M$.
One of the purposes of this paper is to compare the $L^2$ norms of
holomorphic sections of $\LL$ with the $L^2$ norms of holomorphic
sections of this quotient line bundle equipped with the quotient
metric. More explicitly, suppose $G$ is a connected compact Lie
group, $\fg$ its Lie algebra, and $\tau$ a holomorphic Hamiltonian
action of $G$ on $M$ with a proper moment map $\Phi$. Moreover,
assume that there exists a lifting, $\tau^\#$, of $\tau$ to $\LL$,
which preserves the Hermitian inner product $\langle \cdot, \cdot
\rangle$. If the $G$-action on $\Phi^{-1}(0)$ is free, the
quotient space
   \begin{equation*}
   M_{red} = \Phi^{-1}(0)/G
   \end{equation*}
is a compact K\"ahler manifold. Moreover, the Hermitian line
bundle $(\LL, \langle \cdot, \cdot \rangle)$ on $M$ naturally
descends to a Hermitian line bundle $(\LL_{red}, \langle \cdot,
\cdot\rangle_{red})$ on $M_{red}$, and the curvature form of
$\LL_{red}$ is the reduced K\"ahler form $-\omega_{red}$, thus
$\LL_{red}$ is a pre-quantum line bundle over $M_{red}$ (c.f. \S 2
for more details). From these line bundle identifications one gets
a natural map
   \begin{equation}
   \label{eqn:qr}
   \Gamma_{hol}(\LL^k)^G \to \Gamma_{hol}(\LL_{red}^k)
   \end{equation}
and one can prove
\begin{theorem}
[Quantization commutes with reduction for K\"ahler manifolds]
Suppose that for some $k_0 > 0$ the set
$\Gamma_{hol}(\LL^{k_0})^G$ contains a nonzero element. Then the
map (\ref{eqn:qr}) is bijective for every $k$.
\end{theorem}

The proof of this theorem in \cite{GuS82} implicitly involves the
notion of \emph{stability function} and one of the goals of this
article will be to make the role of this function in geometric
invariant theory more explicit. To define this function let
$G_\CC$ be the complexification of $G$ (See \S 2) and let $M_{st}$
be the $G_\CC$ flow-out of $\Phi^{-1}(0)$. Modulo the assumptions
in the theorem above $M_{st}$ is a Zariski open subset of $M$, and
if $G$ acts freely on $\Phi^{-1}(0)$ then $G_\CC$ acts freely on
$M_{st}$ and
   \begin{equation*}
   M_{red} = \Phi^{-1}(0)/ G = M_{st}/G_\CC.
   \end{equation*}
Let $\pi$ be the projection of $M_{st}$ onto $M_{red}$. The
stability function associated to this data is a real-valued
$C^\infty$ map $\psi: M_{st} \to \RR$ with the defining property
   \begin{equation}
   \label{eqn:1.4}
   \langle \pi^*s, \pi^*s \rangle  = e^{\psi} \pi^* \langle s, s
   \rangle_{red}
   \end{equation}
for one or, equivalently, all sections $s$ of our line bundle
$\LL_{red}$. This function can also be viewed as a relative
potential function, relating the K\"ahler form $\omega$ on
$M_{st}$ to the K\"ahler form on $M_{red}$, i.e. $\psi$ satisfies
   \begin{equation}
   \label{eqn:1.5.a}
   \omega - \pi^* \omega_{red} = \sqrt{-1} \partial \bar \partial
   \psi.
   \end{equation}
   %
%modified+%
We will show that this function is proper, non-positive, and takes
its maximum value $0$ precisely on $\Phi^{-1}(0)$. We will also
show that this property suffices to determine it in general by
showing that on the gradient trajectory of any component of $\Phi$
it satisfies a simple ODE. We will then go on to show how to exploit this fact in a number of concrete cases.
%modified-%

Another basic property of $\psi$ is that, for any point $p \in
\Phi^{-1}(0)$, $p$ is the only critical point of the restriction
of $\psi$ to the ``orbit" $\exp{(\sqrt{-1}\fg)} \cdot p$ (Here
$\exp{(\sqrt{-1}\fg)}$ is the ``imaginary" part of $G_\CC$.) Let
$dx$ be the (Riemannian) volume form on this orbit, which is
induced by the restriction to $\exp{(\sqrt{-1}\fg)} \cdot p$ of
the K\"ahler-Riemannian metric on $M_{st}$. By applying the method
of steepest descent, one gets an asymptotic expansion
   \begin{equation}
   \label{eqn:1.5}
   \int_{\exp{(\sqrt{-1}\fg)} \cdot p} f \, e^{\lambda \psi} dx \sim
   \left(\frac{\lambda}{\pi}\right)^{-m/2} \left( 1 +
   \sum_{i=1}^\infty c_i \lambda^{-i} \right)
   \end{equation}
for $\lambda$ large, where $f$ is a smooth function of compact support, $m$ is the dimension of $G$, and $c_i$
are constants depending on the Taylor expansion of $f$ at $p$. (Throughout this paper we will fix
the notations $d=\dim_\CC M, m=\dim_\RR G$ and $n=d-m=\dim_\CC
M_{red}$.)

The asymptotic formula (\ref{eqn:1.5}) has many applications. First, if $f$ is $G$-invariant,
by integrating (\ref{eqn:1.5}) over the $G$-orbit through $p$, we
get
   \begin{equation}
   \label{eqn:1.6}
   \int_{G_\CC \cdot p} f \, e^{\lambda \psi} \frac{\omega^m}{m!} \sim
   \left(\frac{\lambda}{\pi}\right)^{-m/2} V(p) f(p) \left( 1 + O(\frac
   1{\lambda})\right)
   \end{equation}
as $\lambda \to \infty$, where $V(p)$ is the Riemannian volume of
the $G$-orbit through $p$. Integrating this over $M_{red}$ thus gives, for any holomorphic section $s_k
\in \Gamma_{hol}(\LL_{red}^k)$,
   \begin{equation}
   \label{eqn:1.7}
   \left(\frac k{\pi}\right)^{m/2} \|\pi^*s_k\|^2 =
   \|V^{1/2}s_k\|^2_{red} + O(\frac 1k).
   \end{equation}
For more detail, see sections \S \ref{sec:4.1} - \S \ref{sec:4.3} below. This can be viewed as a ``$\frac 12$-form correction" which makes
the identification of $\Gamma_{hol}(\LL^k_{red})$ with
$\Gamma_{hol}(\LL^k)^G$ an isometry modulo $O(\frac 1k)$. (Compare
with \cite{HaK}, \cite{Li} for similar results on $\frac 12$-form
corrections).

A second application of (\ref{eqn:1.5}) concerns the
measures associated with holomorphic sections of $\LL^k_{red}$:
Let $\mu$ and $\mu_{red}$ be the symplectic volume forms on $M$
and $M_{red}$ respectively. Given a sequence of ``quantum states"
   \begin{equation*}
   s_k \in \Gamma_{hol}(\LL^k_{red})
   \end{equation*}
one can, by (\ref{eqn:1.5}), relate the asymptotics of the measures
   \begin{equation}
   \label{eqn:1.8}
   \langle s_k, s_k\rangle \mu_{red}
   \end{equation}
defined by these quantum states for appropriately chosen sequences
of $s_k$'s to the asymptotics of the corresponding measures
   \begin{equation}
   \label{eqn:1.9}
   \langle \pi^* s_k, \pi^* s_k \rangle \mu
   \end{equation}
on $M$. In the special case where $M$ is $\CC^d$ with the flat metric
and $M_{red}$ a toric variety with the quotient metric the
asymptotics of (\ref{eqn:1.9}) can be computed explicitly by Mellin
transform techniques (see \cite{GuW} and \cite{Wan}) and from this
computation together with the identity (\ref{eqn:1.5}) one gets an
alternative proof of the asymptotic properties of (\ref{eqn:1.8})
for toric varieties described in \cite{BGU}.

One can also regard the function
   \begin{equation}
   \label{eqn:1.10}
   \langle s_k, s_k \rangle: M_{red} \to \RR
   \end{equation}
as a random variable and study the asymptotic properties of its
probability distribution, i.e., the measure
   \begin{equation}
   \label{eqn:1.11}
   \langle s_k, s_k \rangle_* \mu_{red},
   \end{equation}
on the real line. These properties, however, can be read off from
the asymptotic behavior of the \emph{moments} of this measure,
which are, by definition just the integrals
   \begin{equation}
   \label{eqn:1.12}
   m_{red}(l, s_k, \mu_{red}) = \int_{M_{red}} \langle s_k, s_k
   \rangle^l d\mu_{red}, \indent l=1, 2, \cdots
   \end{equation}
and by (\ref{eqn:1.5}) the asymptotics of these integrals can be
related to the asymptotics of the corresponding integrals on $M$
viz
   \begin{equation}
   \label{eqn:1.13}
   m(l, \pi^* s_k, \mu) = \int_M \langle \pi^* s_k, \pi^* s_k
   \rangle^l \mu\ .
   \end{equation}
In the toric case Shiffman, Tate and Zelditch showed in \cite{STZ}
that if $s_k$ lies in the weight space
$\Gamma_{hol}(\LL^k)^{\alpha_k}$, where $\alpha_k = k \alpha +
O(\frac 1k)$, and $\nu = {(\Phi_P^* \omega_{FS})^n}/{n!}$ is the
pull-back of the Fubini-Study volume form on the projective space
via the monomial embedding $\Phi_P$, then, if $s_k$ has $L^2$ norm
$1$,
   \begin{equation}
   \label{eqn:1.14}
   \left(\frac k{\pi}\right)^{-n(l-1)/2} m_{red}(l, s_k, \nu)
   \sim \frac{c^l}{l^{n/2}}
   \end{equation}
as $k$ tends to infinity, $c$ being a positive constant. From this
they derived a ``universal distribution law" for such measures. We
will give below a similar asymptotic result for the moments
associated with another volume form, $V\mu_{red}$, which can be
derived from (\ref{eqn:1.5}) and an analogous, but somewhat simpler
version \ of (\ref{eqn:1.14}) for the moments (\ref{eqn:1.13})
upstairs on $\CC^d$.

Related to these results is another application of (\ref{eqn:1.5}):
Let
   \begin{equation*}
   \pi_N: L^2(\LL^N, \mu) \to \Gamma_{hol}(\LL^N)
   \end{equation*}
be the orthogonal projection of the space of $L^2$-sections of
$\LL^N$ onto the space of holomorphic sections of $\LL^N$ and for
$f \in C^\infty(M)$ let
   \begin{equation*}
   M_f: L^2(\LL^N, \mu) \to L^2(\LL^N, \mu)
   \end{equation*}
be the ``multiplication by $f$'' operator. If $M$ is compact
(which will be the case below with $M$ replaced by $M_{red}$) then
by contracting this operator to $\Gamma_{hol}(\LL^N)$ and taking
its trace one gets a measure
   \begin{equation}
   \label{eqn:1.15}
   \mu_N(f) = \tr (\pi_N M_f \pi_N)
   \end{equation}
which one can also write (somewhat less intrinsically) as the
``density of states"
   \begin{equation}
   \label{eqn:1.16}
   \mu_N = \sum \langle s_{N,i}, s_{N,i} \rangle \mu,
   \end{equation}
the $s_{N,i}$'s being an orthonormal basis of
$\Gamma_{hol}(\LL^N)$ inside $L^2(\LL^N,\mu)$. By a theorem of
Boutet de Monvel-Guillemin, $\mu_N(f)$ has an asymptotic
expansion,
   \begin{equation}
   \label{eqn:1.17}
   \mu_N(f) \sim \sum_{i=d-1}^{-\infty} a_i(f) N^i
   \end{equation}
as $N \to \infty$. One of the main results of this paper is a
$G$-invariant version of Boutet de Monvel-Guillemin's result. More
precisely, if we let $\pi_N^G$ be the orthogonal projection
   \begin{equation*}
   \pi_N^G: L^2(\LL^N, \mu) \to \Gamma_{hol}(\LL^N)^G,
   \end{equation*}
then for any $G$-invariant function $f$ on $M$ we have the
asymptotic expansion
   \begin{equation}
   \label{eqn:1.18}
   \mu_N^G(f) = \tr(\pi_N^G M_f \pi_N^G) \sim
   \sum_{i=n-1}^\infty a^G_i(f) N^i
   \end{equation}
as $N \to \infty$. Moreover, the identity (\ref{eqn:1.5}) enables
one to read off this upstairs $G$-invariant expansion from the
downstairs expansion and vice versa. Notice that for this
$G$-invariant expansion, we don't have to require the upstairs
manifold to be compact. For example, for toric varieties, the
upstairs space, $\CC^d$, is not compact, so the space of
holomorphic sections is infinite dimensional, and
Boutet-Guillemin's result doesn't apply; however the $G$-invariant
version of the upstairs asymptotics can, in this case, be computed
directly by Mellin transform techniques (\cite{GuW}, \cite{Wan})
together with an Euler-Maclaurin formula for convex lattice
polytopes (\cite{GuS06}) and hence one gets from (\ref{eqn:1.5}) an
alternative proof of the asymptotic expansion of $\mu_N$ for toric
varieties obtained in \cite{BGU}.

As a last application of the techniques of this paper we discuss
``Bohr-Sommerfeld" issues in the context of GIT theory. Let
$\nabla_{red}$ be the K\"ahlerian connection on $\LL_{red}$ with
defining property,
   \begin{equation*}
   curv(\nabla_{red})=-\omega_{red}.
   \end{equation*}
A Lagrangian submanifold $\Lambda_{red} \subset M_{red}$ is said
to be \emph{Bohr-Sommerfeld} if the connection
$\iota^*_{\Lambda_{red}} \nabla_{red}$ is trivial. In this case
there exists a covariant constant non-vanishing section, $s_{BS}$,
of $\iota^*_{\Lambda_{red}} \LL_{red}$. Viewing $s_{BS}$ as a
``delta section" of $\LL_{red}$ and projecting it onto
$\Gamma_{hol}(\LL_{red})$, one gets a holomorphic section
$s_{\Lambda_{red}}$ of $\LL_{red}$, and one would like to know
\begin{enumerate}
\item Is this section nonzero?
\item What, in fact, is this section?
\item What about the sections $s_{\Lambda_{red}}^{(k)}$ of
$\LL^k_{red}$? Do they have interesting asymptotic properties as
$k \to \infty$? Do they, for instance, ``concentrate"
asymptotically on $\Lambda_{red}$?
\end{enumerate}
We will show that the ``downstairs" versions of these questions on
$M_{red}$ can be translated into ``upstairs" versions of these
questions on $M$ where they often become more accessible.

The paper is organized into three parts. Part 1, \S2-\S5, focuses on the
general theory of stability functions, part 2, \S6-\S8, on
stability functions on toric varieties, and part 3, \S9-\S12, on
stability functions on some non-toric varieties, especially spherical varieties and symplectic quiver varieties.

More precisely, in \S 2 we start with some basic facts about K\"ahler
reduction and geometric invariant theory and define the reduced
line bundle $\LL_{red}$. In \S 3 we prove a number of basic
properties of the stability function, and in \S 4 we derive the
asymptotic expansion (\ref{eqn:1.5}), and use it to show that the
map between $\Gamma_{hol}(\LL^k_{red})$ and
$\Gamma_{hol}(\LL^k)^G$ can be made into an asymptotic isometry by
means of $\frac 12$-forms. Then in \S 5 we deduce from
(\ref{eqn:1.5}) the results about density of states, probability
distributions and Bohr-Sommerfeld sections which we described
above. 

The second part of the paper begins in \S 6 with a review of the Delzant description of toric
varieties as GIT quotients of $\CC^d$ and discuss some of its
implications. In \S 7 we derive an explicit formula, in terms of
moment polytope data, for the stability function involved in this
description and in \S\S 8.1 and 8.2 specialize the results of \S 5 to toric
varieties and discuss their relation to the results of \cite{BGU} alluded to above. In \S 8.3 we show that 
the standard (or ``flat") reduced metrics are just one example of a large class of reduced metrics for which the results of \S 5 hold. In addition, we show that the stability function formula of \S 7 generalizes appropriately to these metrics. 

We make a few remarks here on the generality of the results described here for toric varieties. As already noted, \cite{STZ} contains very interesting results similar to those in this part of the paper, but for a wider class of metrics (the so-called Bergman metrics) than the metric obtained from Delzant's construction by reduction (the flat reduced metric). In fact, in \cite{SoZ} the key step of analyzing the $L^2$-norms of  holomorphic sections of $\LL_{red}^N$ is carried out for arbitrary toric K\"ahler metrics. Since the thrust of \cite{SoZ} was directed towards Donaldson's program on special K\"ahler metrics, this result was not applied at that time to the problems described here, though it clearly could be. The calculations sketched in \S 8.3 show that the reduction methods used here in full only for the flat reduced metric can be extended to the general toric K\"ahler metric \emph{on the open $G_{\CC}$ orbit}. We leave to a later day analyzing the ``boundary condition!
 
 s" along the lower dimensional $G_{\CC}$ strata and reconciling the methods presented here with those of \cite{STZ} and \cite{SoZ} in the general case on a projective toric variety.

Finally in the last part of this paper we take a tentative
first step toward generalizing the results of \S6-\S8 to some
non-abelian analogues of toric varieties: \emph{spherical
varieties} and \emph{symplectic quiver varieties}. The simplest examples of spherical varieties are the
coadjoint orbits of $\mathcal U(n)$ viewed as $\mathcal
U(n-1)$-manifolds. Following Shaun Martin, we will show how these
varieties can be obtained by symplectic reduction from a linear
action of a compact Lie group on $\CC^N$, and, as above for toric
varieties, compute their stability functions. In \S 11, following
Joel Kamnitzer \cite{Kam} we show how Martin's description of
$\mathcal U(n)$-coadjoint orbits extends to quiver varieties, and
for some special classes of quiver varieties (e.g. the polygon
spaces of Kapovich-Millson \cite{KaM}) show that there are
stability function formulae similar to those for coadjoint orbits.
For spherical varieties in general the question still seems to be
open as to whether they have ``nice" descriptions as GIT quotients
analogous to the Delzant description of toric varieties.

\subsection*{Acknowledgement}
The authors are very grateful to the referee for his helpful comments and for a suggested reference.

%##############################
%        Section 2            %
%##############################

\section{K\"ahler reduction vs geometric invariant theory}
\label{sec:2}

\subsection{K\"ahler reduction}
\label{sec:2.1}
Suppose $(M, \omega)$ is a symplectic manifold, $G$ a connected
compact Lie group acting in a Hamiltonian fashion on $M$, and
$\Phi: M \to \fg^*$ a moment map, i.e.,  $\Phi$ is equivariant
with respect to the given $G$-action on $M$ and the coadjoint
$G$-action on $\fg^*$, with the defining property
   \begin{equation}
   \label{eqn:2.1}
   d\langle \Phi, v \rangle = \iota_{v_M}\omega,\qquad v \in \fg,
   \end{equation}
where $v_M$ is the vector field on $M$ generated by the
one-parameter subgroup $\{\exp(-tv)\ |\ t \in \RR\}$ of $G$.
Furthermore we assume that $\Phi$ is proper, $0$ is a regular
value and that $G$ acts freely on the zero level set
$\Phi^{-1}(0)$. Then by the Marsden-Weinstein theorem, the
quotient space
   \begin{equation*}
   M_{red} = \Phi^{-1}(0) / G
   \end{equation*}
is a connected compact symplectic manifold with symplectic form
$\omega_{red}$ satisfying
   \begin{equation}
   \label{eqn:2.2}
   \iota^* \omega = \pi_0^* \omega_{red},
   \end{equation}
where $\iota: \Phi^{-1}(0) \hookrightarrow M$ is the inclusion
map, and $\pi_0: \Phi^{-1}(0) \to M_{red} $ the quotient map.
Moreover, if $\omega$ is integral, so is $\omega_{red}$; and if
$(M, \omega)$ is K\"ahler with holomorphic $G$-action, then
$M_{red}$ is a compact K\"ahler manifold and $\omega_{red}$ is a
K\"ahler form.

\subsection{GIT quotients}
\label{sec:2.2}
The K\"ahler quotient $M_{red}$ also has the following GIT
description:

Let $G_\CC$ be the complexification of $G$, i.e.,  $G_\CC$ is the
unique connected complex Lie group with Lie algebra $\fg_\CC = \fg
\oplus \sqrt{-1}\fg$ which contains $G$ as its maximal compact
subgroup. We will assume that the action of $G$ on $M$ extends
canonically to a holomorphic action of $G_\CC$ on $M$ (This will
automatically be the case if $M$ is compact). The infinitesimal
action of $G_\CC$ on $M$ is given by
   \begin{equation}
   \label{eqn:2.3}
   w_M = J v_M
   \end{equation}
for $v \in \fg, w = \sqrt{-1}v$, where $J$ is the automorphism of
$TM$ defining the complex structure.

The set of stable points, $M_{st}$, of $M$ (with respect to this
$G_\CC$ action) is defined to be the $G_\CC$-flow out of
$\Phi^{-1}(0)$:
   \begin{equation}
   \label{eqn:2.4}
   M_{st} = G_\CC \cdot \Phi^{-1}(0).
   \end{equation}
This is an open subset of $M$ on which $G_\CC$ acts freely, and
each $G_\CC$-orbit in $M_{st}$ intersects $\Phi^{-1}(0)$ in
precisely one $G$-orbit, c.f. \cite{GuS82}. Moreover, for any
$G$-invariant holomorphic section $s_k$ of $\LL^k$, $M_{st}$
contains all $p$ with $s_k(p) \ne 0$. (For a proof, see the
arguments at the end of \S 3.2). In addition, if $M$ is compact
$M-M_{st}$ is just the common zero sets of these $s_k$'s. Since
$M_{st}$ is a principal $G_\CC$ bundle over $M_{red}$, the $G_\CC$
action on $M_{st}$ is proper. The quotient space $M_{st}/G_\CC$
has the structure of a complex manifold. Moreover, since each
$G_\CC$-orbit in $M_{st}$ intersects $\Phi^{-1}(0)$ in precisely
one $G$-orbit, this GIT quotient space coincides with the
symplectic quotient:
   \begin{equation*}
   M_{red} = M_{st} / G_\CC.
   \end{equation*}
In other words, $M_{red}$ is a K\"ahler manifold with
$\omega_{red}$ its K\"ahler form, and the projection map $\pi:
M_{st} \to M_{red}$ is holomorphic.

\subsection{Reduction at the quantum level}
\label{sec:2.3}
Suppose $(\LL, \langle \cdot, \cdot \rangle)$ is a pre-quantum
line bundle over $M$. There is a unique holomorphic connection
$\nabla$ on $\LL$, (called the metric connection), which is
compatible with the Hermitian inner product on $\LL$, i.e.,
satisfies the compatibility condition for every locally
nonvanishing holomorphic section $s: U \to \LL$,
   \begin{equation}
   \label{eqn:2.5}
   \frac{\nabla s}s = \partial \log{\langle s, s \rangle} \in
   \Omega^{1, 0}(U).
   \end{equation}
The pre-quantization condition amounts to requiring that the
curvature form of the connection $\nabla$ is $-\omega$, i.e.,
   \begin{equation}
   \label{eqn:2.6}
   curv(\nabla) := -\sqrt{-1} \bar\partial \partial \log{\langle
   s, s\rangle}= -\omega.
   \end{equation}

To define reduction on the quantum level, we assume that the $G$
action on $M$ can be lifted to an action $\tau^\#$ of $G$ on
$\LL$ by holomorphic line bundle automorphisms. By averaging, we
may assume that $\tau^\#$ preserves the metric $\langle \cdot,
\cdot \rangle$, and thus preserves the connection $\nabla$ and the
curvature form $\omega$. By Kostant's formula (\cite{Kos}), the
infinitesimal action of $\fg$ on sections of $\LL$ is
   \begin{equation}
   \label{eqn:2.7}
   L_v s = \nabla_{v_M}s - \sqrt{-1} \langle \Phi, v \rangle s
   \end{equation}
for all smooth sections $s \in \Gamma(\LL)$ and all $v \in \fg$.
Since $G$ acts freely on $\Phi^{-1}(0)$, the lifted action
$\tau^\#$ is free on $\iota^*\LL$. The quotient
   \begin{equation*}
   \LL_{red} = \iota^* \LL /G
   \end{equation*}
is now a holomorphic line bundle over $M_{red}$.

On the other hand, by \cite{GuS82}, the lifted action $\tau^\#$
can be extended canonically to an action $\tau_\CC^\#$ of $G_\CC$
on $\LL$. Denote by $\LL_{st}$ the restriction of $\LL$ to the
open set $M_{st}$, then $G_\CC$ acts freely on $\LL_{st}$, and we
get the GIT description of the quotient line bundle,
   \begin{equation*}
   \LL_{red} = \LL_{st} / G_{\CC}.
   \end{equation*}
On $\LL_{red}$ there is a naturally defined Hermitian structure,
$\langle \cdot, \cdot \rangle_{red}$, defined by
   \begin{equation}
   \label{eqn:2.8}
   \pi_0^* \langle s, s \rangle_{red}  = \iota^* \langle \pi^*s,
   \pi^* s\rangle
   \end{equation}
for all $s \in \Gamma(\LL_{red})$. Moreover, the induced curvature
form of $\LL_{red}$ is the reduced K\"ahler form $\omega_{red}$.
In other words, the quotient line bundle $(\LL_{red}, \langle
\cdot, \cdot \rangle_{red})$ is a pre-quantum line bundle over the
quotient space $(M_{red}, \omega_{red})$.

%##############################
%        Section 3            %
%##############################

\section{The stability function}
\label{section 3}

\subsection{Definition of the stability function}
\label{sec:3.1}
The stability function $\psi: M_{st} \to \mathbb R$ is defined to
satisfy
   \begin{equation}
   \label{eqn:3.1}
   \langle \pi^*s, \pi^*s \rangle  = e^{\psi} \pi^* \langle s, s
   \rangle_{red}\ .
   \end{equation}
More precisely, suppose $U$ is an open subset in $M_{st}$ and $s:
U \to \LL_{red}$ a non-vanishing section, then $\psi$ restricted to
$\pi^{-1}(U)$ is defined to be
   \begin{equation}
   \label{eqn:3.2}
   \psi = \log{\langle \pi^* s, \pi^* s \rangle} - \pi^*\log{\langle
   s, s\rangle_{red}}\ .
   \end{equation}
Obviously this definition is independent of the choice of $s$.

It is easy to see from the definition that $\psi$ is a $G$-invariant
function on $M_{st}$ which vanishes on $\Phi^{-1}(0)$, and by (2.6),
   \begin{equation}
   \label{eqn:3.3}
   \omega = \pi^* \omega_{red} + \sqrt{-1}\ \bar{\partial} \partial
   \psi.
   \end{equation}
Thus $\psi$ can be thought of as a potential function for the
restriction of $\omega$ to $M_{st}$ \emph{relative} to
$\omega_{red}$.

\begin{remark}
\label{remark:3.1} From the definition it is also easy to see that
the stability function depends on the metric on the line bundle.
One such dependence that is crucial in the whole paper is the
following: If $\LL$ is the trivial line bundle over $\CC$ with the
Bargmann metric and $\psi$ the stability function for some
K\"ahler quotient of $\CC$ associated with $\LL$, then $\LL^N$ is
still the trivial line bundle over $\CC$ but with a slightly
different metric, i.e. the $N^{th}$ tensor of the Bargmann metric,
and the corresponding stability function becomes $N\psi$.
\end{remark}
\begin{remark}
\label{remark:3.2} (Reduction by stages) Let $G=G_1 \times G_2$ be
a product of compact Lie groups $G_1$ and $G_2$. Then by reduction
in stages $M_{red}$ can be identified with
$(M^{(1)}_{red})^{(2)}$, where $M^{(1)}_{red}$ is the reduction of
$M$ with respect to $G_1$ and $(M^{(1)}_{red})^{(2)}$ the
reduction of $M^{(1)}_{red}$ with respect to $G_2$. Let $M_{st}^G$
and $M_{st}^{G_1}$ be the set of stable points in $M$ with respect
to the $G$-action and $G_1$-action respectively, and
$(M_{red}^{(1)})_{st}^{G_2}$ the set of stable points in
$M_{st}^{(1)}$ with respect to the $G_2$-action. Denote by $\pi_1$
the projection of $M_{st}$ onto $M^{(1)}_{red}$. We claim that
$M_{st}^{G} \subset M_{st}^{G_1}$ and
$\pi_1^{-1}((M_{red}^{(1)})_{st}^{G_2}) = M_{st}^G$. The first of
these assertions is obvious and the second assertion follows from
the identification
   \begin{equation*}
   \aligned
   \pi_1^{-1}((M_{red}^{(1)})_{st}^{G_2})
   & = \pi_1^{-1}((G_2)_\CC \bar\Phi_2^{-1}(0))\\
   & = (G_1)_\CC (\pi_1^{-1}((G_2)_\CC \bar\Phi_2^{-1}(0))
       \cap \Phi_1^{-1}(0)) \\
   & = (G_1)_\CC (G_2)_\CC (\pi_1^{-1}(\bar \Phi_2^{-1}(0))
       \cap \Phi_1^{-1}(0)) \\
   & = G_\CC (\Phi_2^{-1}(0) \cap \Phi_1^{-1}(0)) \\
   & = G_\CC \Phi^{-1}(0).
   \endaligned
   \end{equation*}
Thus $\psi = \psi_1 + \pi_1^* \psi^1_2$, where $\psi$ is the
stability function associated with reduction of $M$ by $G$,
$\psi_1$ the stability function associated with reduction of $M$
by $G_1$, and $\psi^1_2$ the stability function associated with
the reduction of $M^{(1)}_{red}$ by $G_2$.
\end{remark}
\begin{remark}
\label{remark:3.3} (Action on product manifolds) As in the
previous remark let $G=G_1 \times G_2$. Let $M_i$, $i=1, 2$, be
K\"ahlerian $G_i$ manifolds and $\LL_i$ pre-quantum line bundles
over $M_i$, satisfying the assumptions in the previous sections.
Denote by $\psi_i$ the stability function on $M_i$ associated to
$\LL_i$. Letting $G$ be the product $G_1 \times G_2$ the stability
function on the $G$-manifold $M_1 \times M_2$ associated with the
product line bundle $\pr_1^*\LL_1 \otimes \pr_2^*\LL_2$ is
$\pr_1^* \psi_1 + \pr_2^* \psi_2$.
\end{remark}

\subsection{Two useful lemmas}
\label{sec:3.2}
Recall that by (\ref{eqn:2.3}), the vector field $w_M$ for the
``imaginary vector" $w =\sqrt{-1}v \in \sqrt{-1}\fg$ is $w_M = J
v_M$.
\begin{lemma}[\cite{GuS82}]
\label{lemma:3.2}
Suppose $w = \sqrt{-1}v \in \sqrt{-1}\fg$, then $w_M$ is the
gradient vector field of $\langle \Phi, v \rangle$ with respect
to the K\"ahler metric $g$.
\end{lemma}
\begin{proof}
   \begin{equation*}
   d\langle \Phi, v\rangle = \iota_{v_M}\omega = \omega(-Jw_M, \cdot)
   =\omega(\cdot, Jw_M) = g(w_M, \cdot).
   \end{equation*}
\end{proof}
\begin{lemma}
\label{lemma:3.3} Suppose $w=\sqrt{-1}v \in \sqrt{-1}\fg$, then
for any nonvanishing $G$-invariant holomorphic section $\tilde s
\in \Gamma_{hol}(\LL)^G$,
   \begin{equation}
   \label{eqn:3.4}
   L_{w_M}\log{\langle \tilde s, \tilde s\rangle} = -2 \langle
   \Phi, v \rangle.
   \end{equation}
\end{lemma}
\begin{proof}
Since
   \begin{equation*}
   J(v_M + \sqrt{-1}w_M) = w_M - \sqrt{-1}v_M = -\sqrt{-1}(v_M +
   \sqrt{-1}w_M),
   \end{equation*}
$v_M + \sqrt{-1}w_M$ is a complex vector field of type (0,1).
Since $\tilde s$ is holomorphic, the covariant derivative
   \begin{equation}
   \label{eqn:3.5}
   \nabla_{v_M} \tilde s = -\sqrt{-1} \nabla_{w_M} \tilde s.
   \end{equation}
Since $\tilde s$ is $G$-invariant, by Kostant's identity
(\ref{eqn:2.7}),
   \begin{equation}
   \label{eqn:3.6}
   0=L_v \tilde s = \nabla_{v_M}\tilde s - \sqrt{-1}\langle \Phi,
                     v\rangle \tilde s.
   \end{equation}
Thus
   \begin{equation*}
   \nabla_{w_M} \tilde s = -\langle \Phi, v \rangle \tilde s.
   \end{equation*}
By metric compatibility, we have for any $G$-invariant holomorphic
section $\tilde s$
   \begin{equation*}
   L_{w_M}\log{\langle \tilde s, \tilde s \rangle} =
   -2 \langle \Phi, v\rangle.
   \end{equation*}
\end{proof}

In particular suppose $M$ is compact, let $\tilde s$ be a
$G$-invariant holomorphic section of $\LL$ and $p$ a point where
$\tilde s(p) \ne 0$. The function
   \begin{equation*}
   \langle \tilde s, \tilde s\rangle: \overline{G_\CC \cdot p} \to \RR
   \end{equation*}
takes its maximum at some point $q$ and since $\overline{G_\CC
\cdot p}$ is $G_\CC$-invariant and
   \begin{equation*}
   \langle \tilde s, \tilde s\rangle (q) \ge \langle \tilde s,
   \tilde s \rangle (p) > 0
   \end{equation*}
it follows from (\ref{eqn:3.4}) that $\Phi(q)=0$, i.e. $q \in
M_{st}$. But $M_{st}$ is open and $G_\CC$-invariant. Hence $p \in
M_{st}$. Thus we've proved that if $p \in M-M_{st}$, then $s(p)=0$
for all $s \in \Gamma_{hol}(\LL)^G$.

\subsection{Analytic properties of the stability function}
\label{sec:3.3}

\begin{proposition}
\label{prop:3.4}
Suppose $w = \sqrt{-1}v \in \sqrt{-1}\fg$, then
$L_{w_M}\psi = -2\langle \Phi, v \rangle$.
\end{proposition}
\begin{proof}
Suppose $s$ is any holomorphic section of the reduced bundle
$\LL_{red}$. Since $\pi^*\log\langle s, s\rangle_{red}$ is
$G_\CC$-invariant, we have from (\ref{eqn:3.2}),
   \begin{equation*}
   L_{w_M}\psi  = L_{w_M}\log\langle \pi^* s, \pi^*
   s\rangle.
   \end{equation*}
Now apply lemma \ref{lemma:3.3} to the $G$-invariant section
$\pi^* s$.
\end{proof}

The main result of this section is
\begin{theorem}
\label{thm:3.5}
$\psi$ is a proper function which takes its
maximum value $0$ on $\Phi^{-1}(0)$. Moreover, for any $p \in
\Phi^{-1}(0)$, the restriction of $\psi$ to the orbit
$\exp{\sqrt{-1}\fg} \cdot p$ has only one critical point, namely
$p$ itself, and this critical point is a global maximum.
\end{theorem}
\begin{proof}
As before we take $w=\sqrt{-1}v \in \sqrt{-1}\fg$. Since $G_\CC$
acts freely on $M_{st}$, we have a diffeomorphism
   \begin{equation}
   \label{eqn:3.7}
   \kappa: \Phi^{-1}(0) \times \sqrt{-1}\fg \to M_{st}, \; (p, w)
   \mapsto \tau_\CC(\exp{w})p.
   \end{equation}
We define two functions
   \begin{equation}
   \label{eqn:3.8}
   \psi_0(p,w,t)=(\kappa^*\psi)(p,tw)
   \end{equation}
and
   \begin{equation}
   \label{eqn:3.9}
   \phi_0(p,w,t)=\langle \kappa^*\Phi(p,tw), v \rangle.
   \end{equation}
Then proposition \ref{prop:3.4} leads to the following
differential equation
   \begin{equation}
   \label{eqn:3.10}
   \frac d{dt}\psi_0 = -2\phi_0,
   \end{equation}
with initial conditions
   \begin{equation}
   \label{eqn:3.11}
   \psi_0(p,w,0)=0
   \end{equation}
and
   \begin{equation}
   \label{eqn:3.12}
   \phi_0(p,w,0)=0.
   \end{equation}

Since $w_M$ is the gradient vector field of $\langle \Phi,
v\rangle$, and $t \mapsto \kappa(p,tw)$ is an integral curve of
$w_M$, we see that $\phi_0$ is a strictly increasing function of
$t$. Thus $\psi_0$ is strictly increasing for $t<0$, strictly
decreasing for $t>0$, and takes its maximal value $0$ at $t=0$.
This shows that $p$ is the only critical point in the orbit
$\sqrt{-1}\fg \cdot p$.

The fact $\psi$ is proper also follows from the differential
equation (\ref{eqn:3.10}), since for any $t_0 > 0$ we have
   \begin{equation*}
   \psi_0(p,w,t) \le C_0 - 2(t-t_0) C_1,\qquad t>t_0
   \end{equation*}
where
   \begin{equation*}
   C_0 = \max_{|w|=1}\psi_0(p,w,t_0)<0
   \end{equation*}
and
   \begin{equation*}
   C_1 = \min_{|w|=1}\phi_0(p,w,t_0)>0.
   \end{equation*}
\end{proof}
\begin{remark}
\label{rem:3.6}
The proof above also gives us an alternate way to
compute the stability function, namely we ``only" need to solve the
differential equation (\ref{eqn:3.10}) along each orbit
$\exp({\sqrt{-1}\fg) \cdot p}$ with initial condition
(\ref{eqn:3.11}). (Of course a much more complicated step is to
write down explicitly the decomposition of $M_{st}$ as a product
$\Phi^{-1}(0) \times \sqrt{-1}\fg$.)
\end{remark}
\begin{corollary}
\label{cor:3.7}
For any $s \in \Gamma_{hol}(\LL_{red})$, the norm
$\langle \pi^* s, \pi^* s\rangle(p)$ is bounded on $M_{st}$, and
tends to $0$ as $p$ goes to the boundary of $M_{st}$.
\end{corollary}

\subsection{Quantization commutes with reduction}
\label{sec:3.4}
As we have mentioned in the introduction, the properties of the
stability function described above were implicitly involved in the
proof of the ``quantization commutes with reduction" theorem in
\cite{GuS82}. We end this section by briefly describing this
proof. Assume $M$ compact. Then using elliptic operator techniques
one can prove that there exists a non-vanishing $G_\CC$-invariant
holomorphic section $\tilde s$ of $\LL^k$ for $k$ large. But
$M_{st}$ contains all points $p$ with $\tilde s(p) \ne 0$. So the
complement of $M_{st}$ is contained in a codimension one complex
subvariety of $M$. By the corollary above, we see that for any
holomorphic section $s$ of $\LL_{red}$, $\pi^* s$ can be extended
to a holomorphic section of $\LL$ by setting $\pi^* s = 0$ on $M -
M_{st}$. This gives the required bijection. For details, c.f.
\cite{GuS82}.

%##############################
%        Section 4            %
%##############################

\section{Asymptotic properties of the stability function}
\label{sec:4}

\subsection{The basic asymptotics}
\label{sec:4.1}
>From the previous section we have seen that the stability function
takes its global maximum $0$ exactly at $\Phi^{-1}(0)$. Thus as
$\lambda$ tends to infinity, $e^{\lambda \psi}$ tends to 0
exponentially fast off $\Phi^{-1}(0)$. So in principle, only a
very small neighborhood of $\Phi^{-1}(0)$ will contribute to the
asymptotics of the integral
   \begin{equation}
   \label{eqn:AIdisplay}
   \int_{M_{st}} f e^{\lambda \psi} \frac{\omega^d}{d!}
   \end{equation}
for $f$ a bounded function in $C^\infty(M_{st})^G$ and for
$\lambda$ large. In this section we will derive an asymptotic
expansion in $\lambda$ for this integral, beginning with
(\ref{eqn:1.5}). We first note that, for proving asymptotic formulas in $\lambda$ for the integral (\ref{eqn:AIdisplay}), one can without loss of generality assume that $f$ in (\ref{eqn:AIdisplay}) is of compact support. This is because, if $f e^{\lambda \psi} \in L^1(M_{st}, \omega^d/d! )$, for $\lambda \geq \lambda_0$, then
\begin{equation}
   \label{eqn:decompAI}
	 \int_{M_{st}} f e^{\lambda \psi} dx = \int_{\{\psi \geq c\}} f e^{\lambda \psi} dx  +  \int_{\{\psi \leq c\}} f e^{\lambda \psi} dx
\end{equation}
where we have abbreviated $\omega^d/d! = dx$, the Riemannian volume form on $M_{st}$, and $c < 0$ is any constant. The first integral on the right in (\ref{eqn:decompAI}) is compactly supported, while the second integral is readily bounded:
\begin{equation}
\label{eqn:expsmall}
\int_{\{\psi \leq c\}} f e^{\lambda \psi} dx \leq C \, e^{-c \lambda}.
\end{equation}

The proof of (\ref{eqn:1.5}) is based on the following method of
steepest descent: Let $X$ be an $m$-dimensional Riemannian
manifold with volume form $dx$, $\psi: X \to \RR$ a real-valued
smooth function which has a unique maximum $\psi(p)=0$ at a point
$p$, and is bounded away from zero outside a compact set. Suppose
moreover that $p$ is a nondegenerate critical point of $\psi$. If
$f \in C^\infty(X)$ is compactly supported, then
   \begin{equation}
   \label{eqn:4.1}
   \int_X f(x)e^{\lambda \psi(x)} dx \sim
   \sum_{k=0}^\infty c_k \lambda^{-\frac m2 -k}, \indent \mbox{as\
   }\lambda \to \infty
   \end{equation}
where the $c_k$'s are constants. Moreover,
   \begin{equation}
   \label{eqn:4.2}
   c_0 = (2\pi)^{m/2} \tau_p f(p),
   \end{equation}
where
   \begin{equation}
   \label{eqn:4.3}
   \tau_p^{-1} = \frac{(\det d^2\psi_p(e_i, e_j))^{1/2}}
   {|dx_p(e_1, \cdots, e_n)|}
   \end{equation}
for any basis $e_1, \cdots, e_m$ of $T_pM$.

>From this general result we obtain:
\begin{theorem}
\label{thm:4.1} Let $dx$ be the Riemannian volume form on
$\exp{(\sqrt{-1}\fg)}\cdot p$ induced by the restriction to
$\exp{(\sqrt{-1}\fg)}\cdot p$ of the K\"ahler-Riemannian metric on
$M_{st}$, where $p$ is any point in $\Phi^{-1}(0)$. Let $f$ be a
smooth function on $M$, compactly supported in $M_{st}$. Then for $\lambda$ large,
   \begin{equation}
   \label{eqn:4.4}
   \int_{\exp{\sqrt{-1}\fg}\cdot p}f(x) e^{\lambda \psi(x)} dx
   \sim \left(\frac{\lambda}{\pi}\right)^{-m/2} \left(f(p) +
   \sum_{i=1}^\infty c_i \lambda^{-i}\right),
   \end{equation}
where $c_i$ are constants depending on $f, \psi$ and $p$.
\end{theorem}
\begin{proof}
We need to compute the Hessian of $\psi$ restricted to
$\exp{(\sqrt{-1}\fg)}\cdot p$ at the point $p$. By proposition
\ref{prop:3.4},
   \begin{equation*}
   d(d\psi(w_M)) = d(L_{w_M}\psi) = -2d\langle \Phi, v\rangle = -2
   \omega(v_M, \cdot),
   \end{equation*}
so
   \begin{equation*}
   d^2\psi_p(w_M, w'_M) = -2\omega_p(v_M, w'_M) = -2g_p(v_M, v'_M)
   = -2 g_p(w_M, w'_M).
   \end{equation*}
This implies
$\tau_p = 2^{-m/2}$.
\end{proof}

\subsection{Asymptotics on submanifolds of $M_{st}$}
\label{sec:4.2}
>From (\ref{eqn:4.4}) we obtain asymptotic formulas similar to
(\ref{eqn:4.4}) for submanifolds of $M_{st}$ which are foliated by
the sets $\exp{(\sqrt{-1}\fg)} \cdot p$. For example, by the
Cartan decomposition
   \begin{equation*}
   G_\CC = G \times \exp{(\sqrt{-1}\fg)}
   \end{equation*}
one gets a splitting
   \begin{equation*}
   G_\CC \cdot p = G \times \exp{(\sqrt{-1}\fg)}\cdot p.
   \end{equation*}
Moreover, this is an orthogonal splitting on $\Phi^{-1}(0)$.
Thus if we write
   \begin{equation*}
   \frac{\omega^m}{m!}(x) = g(x)d\nu \wedge dx,
   \end{equation*}
where $d\nu$ is the Riemannian volume form on the $G$-orbit $G
\cdot p$, defined by the K\"ahler-Riemannian metric, we see that
$g(x)$ is $G$-invariant and $g(p)=1$ on $\Phi^{-1}(0)$. Thus if we
apply theorem \ref{thm:4.1} to a $G$-invariant $f$ we get
\begin{corollary}
\label{cor:4.2}
As $\lambda \to \infty$,
   \begin{equation}
   \label{eqn:4.5}
   \int_{G_{\CC} \cdot p} f(x)e^{\lambda \psi} \frac{\omega^m}{m!}
   \sim V(p) \left(\frac{\lambda}{\pi}\right)^{-m/2}
   \left(f(p) +\sum_{i=1}^\infty c_i(p)\lambda^{-i}\right)\ ,
   \end{equation}
where $V(p)=\Vol(G \cdot p)$ is the Riemannian volume of the $G$
orbit through $p$.
\end{corollary}

Similarly the diffeomorphism (\ref{eqn:3.7}) gives a splitting of
$M_{st}$ into the imaginary orbits $\exp{(\sqrt{-1}\fg)} \cdot p$,
and by the same argument one gets
\begin{corollary}
\label{cor:4.3}
If $e^{\lambda \psi} \in L^1(M_{st}, dx)$, then as $\lambda \to \infty$,
   \begin{equation}
   \label{eqn:4.6}
   \int_{M_{st}} e^{\lambda \psi} dx
   \sim \Vol(\Phi^{-1}(0)) \left(\frac
   {\lambda}{\pi}\right)^{-m/2}  \left(1+\sum_{i=1}^\infty C_i
   \lambda^{-i}\right)\ .
   \end{equation}
\end{corollary}

\subsection{The half form correction}
\label{sec:4.3}
Now we apply corollary \ref{cor:4.2} to prove (\ref{eqn:1.7}).
Since $M_{red} = M_{st}/G_\CC$, we have a decomposition of the
volume form
   \begin{equation}
   \label{eqn:4.7}
   \frac{\omega^d}{d!}  = \pi^*\frac{\omega_{red}^{n}}{n!}
   \wedge d\mu_\pi,
   \end{equation}
where
   \begin{equation*}
   d\mu_\pi(x) =h(x) \frac{\omega^m}{m!},
   \end{equation*}
with $h(p)=1$ on $\Phi^{-1}(0)$. Now suppose $s_k \in
\Gamma_{hol}(\LL_{red}^k)$, and again, that $e^{\lambda \psi} \in L^1(M_{st}, dx)$. Since the stability function of
$\LL_{red}^k$ is $k\psi$,  (\ref{eqn:1.4}) becomes
   \begin{equation*}
   \langle \pi^*s_k, \pi^*s_k \rangle = e^{k\psi} \pi^* \langle s_k,
   s_k \rangle_{red}.
   \end{equation*}
By (\ref{eqn:4.5}),
   \begin{equation*}
   \aligned
   \|\pi^*s_k\|^2 &= \int_{M_{red}} \left(\int_{G_\CC \cdot p}
      e^{k\psi} d\mu_\pi \right)
      \langle s_k, s_k\rangle_{red} \frac{\omega_{red}^{n}}{n!} \\
   & = \left(\frac k{\pi}\right)^{-m/2} \left(1+O(k^{-1})\right)
      \int_{M_{red}} V(\pi_0^{-1}(q)) \langle s_k, s_k\rangle_{red}
      \frac{\omega_{red}^{n}}{n!}.
   \endaligned
   \end{equation*}
In other words,
   \begin{equation}
   \label{eqn:4.8}
   \left(\frac k{\pi}\right)^{m/2} \|\pi^* s_k\|^2 =
   \|V^{1/2}s_k\|_{red}^2 + O(\frac 1k),
   \end{equation}
where $V$ is the volume function $V(q) = V(\pi^{-1}_0(q))$.

The presence of the factor $V$ can be viewed as a ``$\frac
12$-form correction" in the Kostant-Souriau version of geometric
quantization. Namely, let $\KK = \bigwedge^d(T^{1,\,0}M)^*$ and
$\KK_{red} = \bigwedge^n(T^{1,\,0}M_{red})^*$ be the canonical
line bundles on $M$ and $M_{red}$ and let $\ll, \gg$ and $\ll,
\gg_{red}$ be the Hermitian inner products on these bundles, then
   \begin{equation*}
   \pi_0^* \KK_{red} = \iota^* \KK
   \end{equation*}
and
   \begin{equation*}
   \pi_0^* (V \ll, \gg_{red}) = \iota^* \ll, \gg.
   \end{equation*}
So if $\KK^{\frac 12}$ and $\KK_{red}^{\frac 12}$ are ``$\frac
12$-form" bundles on $M$ and $M_{red}$ (i.e., the square roots of
$\KK$ and $\KK_{red}$), then one has a map
   \begin{equation*}
   \Gamma_{hol}(\LL^k \otimes \KK^{\frac 12})^G \to \Gamma_{hol}
   (\LL^k_{red} \otimes \KK_{red}^{\frac 12})
   \end{equation*}
which is an isometry modulo an error term of order $O(k^{-1})$.
(See \cite{HaK} and \cite{Li} for more details.)

%##############################
%        Section 5            %
%##############################

\section{Applications to spectral problems on K\"ahler
quotients}
\label{sec:5}

\subsection{Maximum points of quantum states}
\label{sec:5.1}
Suppose $M$ is a K\"ahler manifold with quantum line bundle $\LL$,
and $\tilde s \in \Gamma_{hol}(\LL)$ is a quantum state. The
``invariance of polarization" conjecture of Kostant-Souriau is
closely connected with the question: where does the function
$\langle \tilde s, \tilde s \rangle$ take its maximum? If $C$ is
the set where $\langle \tilde s, \tilde s\rangle$ takes its
maximum, what can one say about $C$? What is the asymptotic
behavior of the function $\langle \tilde s, \tilde s\rangle^k$ in
a neighborhood of $C$?

To address these questions we begin by recalling the following
results:
\begin{proposition}
\label{prop:5.1}
If $C$ above is a submanifold of $M$, then \\
$\mathrm{(a)}$ $C$ is an isotropic submanifold
of $M$;\\
$\mathrm{(b)}$ $\iota_C^* \tilde s$ is a non-vanishing covariant
constant section of $\iota_C^* \LL$;\\
$\mathrm{(c)}$ Moreover if $M$ is a K\"ahler $G$-manifold and
$\tilde s$ is in $\Gamma_{hol}(\LL)^G$ then $C$ is contained in
the zero level set of $\Phi$.
\end{proposition}
\begin{proof}
$\mathrm{(a)}$ Let $\alpha = \sqrt{-1} \bar \partial \log \langle
\tilde s, \tilde s \rangle$. Then $\omega = d \alpha$ and
$\alpha_p = 0$ for every $p \in C$, so $\iota_C^* \omega = 0$.
\\
$\mathrm{(b)}$ By (\ref{eqn:2.5}), $\nabla s = 0$ on C.
\\
$\mathrm{(c)}$ By (\ref{eqn:2.7}),
   \begin{equation*}
   \nabla_{v_M}s = \sqrt{-1} \langle \Phi, v \rangle s = 0
   \end{equation*}
along $C$, therefore since $s$ is non-zero on $C$,
$\langle \Phi, v \rangle = 0$ on $C$.
\end{proof}

We will call a submanifold $C$ of $M$ for which the line bundle
$\iota_C^* \LL$ admits a nonzero covariant constant section a
\emph{Bohr-Sommerfeld} set. Notice that if $s_0$ is a section of
$\iota_C^* \LL$ which is non-vanishing, then
   \begin{equation*}
   \frac{\nabla s_0}{s_0} = \alpha_0 \Leftrightarrow
   d\alpha_0 = \iota_C^* \omega ,
   \end{equation*}
so if $s$ is covariant constant then $C$ has to be isotropic. The
most interesting Bohr-Sommerfeld sets are those which are
maximally isotropic, i.e., Lagrangian, and the term
``Bohr-Sommerfeld" is usually reserved for these Lagrangian
submanifolds .

A basic problem in Bohr-Sommerfeld theory is obtaining converse
results to the proposition above. Given a Bohr-Sommerfeld set,
$C$, does there exist a holomorphic section, $s$, of $\LL$ taking
its maximum on $C$, i.e., for which the measure
   \begin{equation}
   \label{eqn:5.0}
   \langle s^k, s^k \rangle \mu_{Liouville}
   \end{equation}
becomes more and more concentrated on $C$ as $k \to \infty$. As we
pointed out in the introduction this problem is often intractable,
however if we are in the setting of GIT theory with $M$ replaced
by $M_{red}$, then the downstairs version of this question can be
translated into the upstairs version of this question which is
often easier. In \S 5.2 we will discuss the behavior of measures
of type (\ref{eqn:5.0}) in general and then in \S 5.5 discuss this
Bohr-Sommerfeld problem.

\subsection{Asymptotics of the measures (1.10)}
\label{sec:5.2}
We will now apply stability theory to the measure
(\ref{eqn:1.8}) on $M_{red}$. For $f$ an integrable function on
$M_{red}$, consider the asymptotic behavior of the integral
   \begin{equation}
   \label{eqn:5.1}
   \int_{M_{red}} f \langle s_k, s_k \rangle \mu_{red},
   \end{equation}
with $s_k \in \Gamma_{hol}(\LL^k_{red})$ and $k \to \infty$. It is
natural to compare (\ref{eqn:5.1}) with the upstairs integral
   \begin{equation}
   \label{eqn:5.2}
   \int_{M_{st}} \pi^* f \langle \pi^* s_k, \pi^* s_k \rangle \mu.
   \end{equation}
However, since $M_{st}$ is noncompact, the integral above may not
converge in general. To eliminate the possible convergence issues,
we multiply the integrand by a cutoff function, i.e.,  a compactly
supported function $\chi$ which is identically 1 on a neighborhood
of $\Phi^{-1}(0)$. In other words, we consider the integral
   \begin{equation}
   \label{eqn:5.3}
   \int_{M_{st}} \chi \pi^*f \langle \pi^*s_k, \pi^* s_k\rangle
   \mu.
   \end{equation}
Obviously, different choices of the cutoff function will not affect
the asymptotic behavior of (\ref{eqn:5.3}).

Using the decomposition (\ref{eqn:4.7}) we get
   \begin{equation*}
   \aligned
   \int_{M_{st}} \chi \pi^*f \langle \pi^* s_k,
   \pi^* s_k \rangle \frac{\omega^d}{d!}
   & = \int_{M_{red}} \left(\int_{G_\CC \cdot p}
   e^{k\psi}\chi d\mu_{\pi}\right) f \langle s_k, s_k
   \rangle_{red} \frac{\omega_{red}^{n}}{n!} \\
   & \sim \int_{M_{red}} V f \langle s_k, s_k \rangle d\mu_{red},
   \endaligned
   \end{equation*}
where $V(q):=V(\pi^{-1}(q))$ is the volume function. We conclude
\begin{proposition}
\label{prop:5.2}
As $k \to \infty$ we have
   \begin{equation*}
   \int_{M_{red}} f \langle s_k, s_k \rangle \mu_{red}
   \sim (\frac k{\pi})^{-m/2}
   \int_M \chi \tilde f \tilde V^{-1}
   \langle \pi^* s_k, \pi^* s_k \rangle \mu,
   \end{equation*}
where $\tilde f = \pi^{*}f, \tilde V=\pi^* V$ and $\chi$ is any
cutoff function near $\Phi^{-1}(0)$.
\end{proposition}

Similarly if we apply the same arguments to the density of states
   \begin{equation}
   \label{eqn:5.4}
   \mu_N  = \sum_i \langle s_{N, i}, s_{N, i} \rangle \mu_{red},
   \end{equation}
where $\{s_{N,i}\}$ is an orthonormal basis of $\LL^N_{red}$, we
get
\begin{proposition}
\label{prop:5.3}
As $N \to \infty$,
   \begin{equation}
   \label{eqn:5.5}
   \int_{M_{red}}f\mu_N \sim (\frac N{\pi})^{-m/2}
   \int_{M_{st}} \chi \tilde f \tilde V^{-1} \mu_N^G,
   \end{equation}
where $\mu_N^G=\sum\limits_i \langle \pi^* s_{N,i}, \pi^* s_{N,i}
\rangle \mu$ is the upstairs $G$-invariant measure
(\ref{eqn:1.18}).
\end{proposition}

\subsection{Asymptotics of the moments}
\label{sec:5.3}
We next describe the role of ``upstairs" versus ``downstairs" in
describing the asymptotic behavior of the distribution function
   \begin{equation}
   \label{eqn:5.6}
   \sigma_{k}([t, \infty)) = \Vol \{z\ |\ \langle s_{k}, s_k
   \rangle(z) \ge t\},
   \end{equation}
for $s_k \in \Gamma_{hol}(\LL^k_{red})$, i.e.,  of the
push-forward measure, $\langle s_k, s_k \rangle_{*} \mu$, on the
real line $\RR$. The moments (\ref{eqn:1.12}) completely determine
this measure, and by theorem \ref{thm:4.1} the moments
(\ref{eqn:1.12}) on $M_{red}$ are closely related to the
corresponding moments (\ref{eqn:1.13}) on $M$. In fact, by
corollary \ref{cor:4.2} and the decomposition (\ref{eqn:4.7}),
   \begin{equation*}
   \aligned
       \int_{M_{st}} \langle \pi^* s_k, \pi^* s_k \rangle^l \mu
   & = \int_{M_{st}}  (\pi^* \langle s_k, s_k\rangle)^l e^{lk\psi}
       \pi^* \frac{\omega_{red}^n}{n!}\wedge h(x)\frac{\omega^m}
       {m!}\\
   & \sim \left(\frac{lk}{\pi}\right)^{-m/2} \int_{M_{red}}
          \langle s_k, s_k\rangle^l V \mu_{red}.
   \endaligned
   \end{equation*}
We conclude
\begin{proposition}
\label{prop:5.4} For any integer $l$, the $l^{th}$ moments
(\ref{eqn:1.13}) satisfy
   \begin{equation}
   \label{eqn:5.7}
   m(l, \pi^* s_k, \mu) \sim \left(\frac {lk}{\pi}
   \right)^{-m/2} m_{red}(l, s_k, V\mu_{red}).
   \end{equation}
as $k \to \infty$.
\end{proposition}

\subsection{Asymptotic expansion of the $G$-invariant density of states}
\label{sec:5.4}
For the measure (\ref{eqn:1.15}), Boutet-Guillemin showed
that it admits an asymptotic expansion (\ref{eqn:1.17}) in inverse
power of $N$ as $N \to \infty$ if the manifold is compact (See the
appendix for a proof of this result). By applying stability theory
above, we get from the Boutet-Guillemin's expansion for the
downstairs manifold a similar asymptotic expansion upstairs for
the $G$-invariant density of states without assuming $M$ to be
compact. Namely, since $M_{red}$ is compact, Boutet-Guillemin's
theorem gives one an asymptotic expansion
   \begin{equation*}
   \mu_N^{red}(f) = \tr(\pi_N^{red}M_f\pi_N^{red}) \sim
   \sum_{i=n-1}^{-\infty} a_i^{red} N^i,
   \end{equation*}
and for $\pi_N^G: L^2(\LL^N, \mu) \to \Gamma_{hol}(\LL^N)^G$ the
orthogonal projection onto $G$-invariant holomorphic sections, we
will deduce from this:
\begin{theorem}
\label{thm:5.5}
For any compactly supported $G$-invariant function $f$ on $M$,
   \begin{equation}
   \label{eqn:5.8}
   \tr(\pi_N^G M_f \pi_N^G) \sim \sum_{i=n-1}^{-\infty}
   a_i^G(f) N^i,
   \end{equation}
as $N \to \infty$, and the coefficients $a_i^G$ can be computed
explicitly from $a_i^{red}$. In particular, the leading
coefficient $a_{n-1}^G(f)=a_{n-1}^{red}(f_0 V)$, where
$f_0(p)=f(\pi_0^{-1}(p))$.
\end{theorem}
\begin{proof}
Let $\{s_{N,j}\}$ be an orthonormal basis of
$\Gamma_{hol}(\LL^N_{red})$ with respect to the volume form
$V\mu_{red}$, then $\{\pi^* s_{N,j}\}$ is an orthogonal basis of
$\Gamma_{hol}(\LL^N)^G$, and
   \begin{equation*}
   \tr(\pi_N^G M_f \pi_N^G) = \int_M \sum_j \frac{\langle \pi^*
   s_{N,j}, \pi^* s_{N,j} \rangle}{\|\pi^*s_{N,j}\|^2} f \mu,
   \end{equation*}
where, by the same argument as in the proof of (\ref{eqn:4.8}), we
have
   \begin{equation*}
   \|\pi^*s_{N,j}\|^2 \sim \left(\frac N{\pi}\right)^{-m/2}
   \left( 1 + \sum_i C_i N^{-i}\right),
   \end{equation*}
which implies
   \begin{equation*}
   \frac 1{\|\pi^*s_{N,j}\|^2} \sim \left(\frac N{\pi}\right)^{m/2}
   \left( 1 + \sum_i \tilde C_i N^{-i}\right).
   \end{equation*}
Moreover, it is easy to see that
   \begin{equation*}
   \int_{M_{red}} \sum_j \langle s_{N,j}, s_{N,j} \rangle
   Vf_0\mu_{red}
   =\mu_N^{red}(f_0V).
   \end{equation*}
Now the theorem follows from straightforward computations
   \begin{equation*}
   \aligned
  & \tr(\pi_N^G M_f \pi_N^G)
   \sim  \left(\frac N{\pi}\right)^{\frac m2} ( 1 + \sum_i \tilde C_i N^{-i})
          \int_{M_{st}}\sum_j \langle
        \pi^*s_{N,j},\pi^* s_{N,j}\rangle f\mu  \\
   & \sim  \left(\frac N{\pi}\right)^{\frac m2}  ( 1 + \sum_i \tilde C_i N^{-i} )
          \int_{M_{red}}
          \left( \int_{G\cdot p} \sum_{i=-\frac m2}^{-\infty}
           N^i c_i(f, p) \right) \sum_j \langle s_{N,j}, s_{N,j}\rangle
           (p) & \\
   &  = \left(\frac N{\pi}\right)^{\frac m2} ( 1 + \sum_i \tilde C_i N^{-i})
         \int_{M_{red}} \sum_{i=-\frac m2}^{-\infty} \left( N^i
        \int_{G \cdot p} c_i(f, p)d\nu\right) \sum_j \langle
        s_{N,j},
        s_{N,j} \rangle (p)\\
   & = \left(\frac N{\pi}\right)^{\frac m2} ( 1 + \sum_i \tilde C_i N^{-i})
        \sum_{i=-\frac m2}^{-\infty} N^{i} \mu_N^{red}
       ( c_i V) \\
   & \sim \sum_{i=n-1}^{-\infty} a_i^G(f) N^i,
   \endaligned
   \end{equation*}
where we used the fact that since $f$ is $G$-invariant, so is
$c_i(f,p)$. This proves (\ref{eqn:5.8}). Moreover, since
$c_{-m/2}(f,p)=f(p)/\pi^{m/2}$, we see that
   \begin{equation*}
   a^G_{n-1}(f) = a_{n-1}^{red}(f_0V),
   \end{equation*}
completing the proof.
\end{proof}

\subsection{Bohr-Sommerfeld Lagrangians}
\label{sec:5.5}
We assume we are in the same setting as before, and denote by
$\nabla_{red}$ the
metric connection on $\LL_{red}$. Suppose
$\Lambda_{red}$ is a Bohr-Sommerfeld Lagrangian submanifold of
$M_{red}$, and $s_{BS}$ is a covariant constant section, i.e.,
   \begin{equation}
   \label{eqn:5.9}
   s_{BS}: \Lambda_{red} \to \iota_{\Lambda_{red}}^*\LL_{red}, \quad
   (\iota_{\Lambda_{red}}^*\nabla_{red}) s_{BS} = 0,
   \end{equation}
where $\iota_{\Lambda_{red}}: \Lambda_{red} \to M_{red}$ is the
inclusion map. Let $\Lambda = \pi_0^{-1}(\Lambda_{red})$, then
$\Lambda \subset \Phi^{-1}(0)$ is a $G$-invariant Lagrangian
submanifold of $M$. Since
   \begin{equation}
   \label{eqn:5.10}
   \pi_0^* \nabla_{red} s_{BS} = \iota_\Lambda^* \nabla \pi_0^* s_{BS},
   \end{equation}
we see that $\pi_0^* s_{BS}$ is a covariant constant section on
$\Lambda$. In other words, $\Lambda$ is a Bohr-Sommerfeld
Lagrangian submanifold of $M$. Conversely, if $\Lambda$ is a
$G$-invariant Bohr-Sommerfeld Lagrangian submanifold of $M$, then
$\Lambda_{red} = \pi_0(\Lambda)$ is a Bohr-Sommerfeld Lagrangian
submanifold of $M_{red}$.

Fixing a volume form $\mu_\Lambda$ on $\Lambda$, the pair
$(\Lambda_{red}, s_{BS})$ defines a functional $l$ on the space of
holomorphic sections by
   \begin{equation*}
   l: \Gamma_{hol}(\LL_{red}) \to \CC, \quad s \mapsto
   \int_{\Lambda_{red}} \langle \iota_{\Lambda_{red}}^*s,
   s_{BS} \rangle \mu_{\Lambda_{red}}.
   \end{equation*}
This in turn defines a global holomorphic section
$s_{\Lambda_{red}} \in \Gamma_{hol}(\LL_{red})$ by duality. In
other words, $s_{\Lambda_{red}}$ is the holomorphic section on
$M_{red}$ with the defining property
   \begin{equation}
   \label{eqn:5.11}
   \int_{M_{red}} \langle s, s_{\Lambda_{red}} \rangle\mu_{red}=
   \int_{\Lambda_{red}} \langle \iota_{\Lambda_{red}}^* s,
   s_{BS} \rangle \mu_{\Lambda_{red}}
   \end{equation}
for all $s \in \Gamma_{hol}(\LL_{red})$. A fundamental problem in
Bohr-Sommerfeld theory is to know whether the section
$s_{\Lambda_{red}}$ vanishes identically; and if not, to what
extent $s_{\Lambda_{red}}$ is ``concentrated" on the set
$\Lambda_{red}$. One can also ask this question for the analogous
section of $\LL_{red}^k$.

We apply the upstairs-vs-downstairs philosophy to these problems.
For the upstairs Bohr-Sommerfeld Lagrangian pair $(\Lambda, \tilde
s_{BS})$, $\tilde s_{BS}=\pi_0^* s_{BS}$, as above one can
associate with it a functional $\tilde l$ on
$\Gamma_{hol}(\LL)^G$, which by duality defines a global
$G$-invariant section $\tilde s_{\Lambda} \in
\Gamma_{hol}(\LL)^G$. Obviously $l \ne 0$ if and only if $\tilde
l$ is nonzero on $\Gamma_{hol}(\LL)^G$. However, since $\tilde
s_{BS}$ is a $G$-invariant section,
   \begin{equation*}
   \langle \tilde s, \tilde s_{BS} \rangle =  \langle \tilde s^G,
   \tilde s_{BS} \rangle ,
   \end{equation*}
where $\tilde s^G$ is the orthogonal projection of $\tilde s \in
\Gamma_{hol}(\LL)$ onto $\Gamma_{hol}(\LL)^G$. Thus $\tilde l$ is
nonzero on $\Gamma_{hol}(\LL)^G$ if and only if it is nonzero on
$\Gamma_{hol}(\LL)$.
Thus we proved
\begin{proposition}
\label{prop:5.6}
$s_{\Lambda_{red}} \ne 0$ if and only if $\tilde s_\Lambda \ne 0$.
\end{proposition}

A natural question to ask is whether $\pi^*s_{\Lambda_{red}}$
coincides with $\tilde s_\Lambda$ on $M_{st}$, or alternatively,
whether $\pi_0^* s_{\Lambda_{red}} = \iota^* \tilde s_\Lambda$ on
$\Phi^{-1}(0)$. In view of the $\frac 12$-form correction, we will
modify the definition of the downstairs section
$s_{\Lambda_{red}}$ to be
   \begin{equation}
   \label{eqn:5.12}
   \int_{M_{red}} \langle s, s_{\Lambda_{red}} \rangle V \mu_{red}=
   \int_{\Lambda_{red}} \langle \iota_{\Lambda_{red}}^* s,
   s_{BS} \rangle V \mu_{\Lambda_{red}}\ ,
   \end{equation}
for $s, s_{\Lambda_{red}} \in \Gamma_{hol}(\LL_{red})$. The
upstairs version of this is
   \begin{equation}
   \label{eqn:5.13}
   \int_{M_{st}} \langle \tilde s, \tilde s_{\Lambda} \rangle \mu=
   \int_{\Lambda} \langle \iota_{\Lambda}^* \tilde s, \pi_0^*
   s_{BS} \rangle \mu_\Lambda
   \end{equation}
for $\tilde s = \pi^* s$. Since $\Lambda =
\pi_0^{-1}(\Lambda_{red})$, the right hand sides of
(\ref{eqn:5.12}) and (\ref{eqn:5.13}) coincide. Thus
   \begin{equation}
   \label{eqn:5.14}
   \int_{M_{st}} \langle \pi^* s, \tilde s_{\Lambda} \rangle \mu=
   \int_{M_{red}} \langle s, s_{\Lambda_{red}} \rangle V \mu_{red}
   \end{equation}
for all $s \in \Gamma_{hol}(\LL_{red})$.

Now we assume $s_k \in \Gamma_{hol}(\LL_{red}^k)$, $s_{BS}^k$ being
the $k^{th}$ tensor power of $s_{BS}$, and let
$s^{(k)}_{\Lambda_{red}}$ and $\tilde s^{(k)}_{\Lambda}$ be the
corresponding holomorphic sections. Then equation (\ref{eqn:5.14})
now reads
   \begin{equation}
   \label{eqn:5.15}
   \int_{M_{st}} \langle \pi^* s_k, \tilde s^{(k)}_{\Lambda}
   \rangle \mu= \int_{M_{red}} \langle s_k, s^{(k)}_{\Lambda_{red}}
   \rangle V \mu_{red}
   \end{equation}
for all $s_k \in \Gamma_{hol}(\LL_{red}^k)$ (However, the sections
$\tilde s^{(k)}_{\Lambda}$ and $s^{(k)}_{\Lambda_{red}}$ are no
longer the $k^{th}$ tensor powers of $\tilde s_{\Lambda}$ and
$s_{\Lambda_{red}}$ above). Notice that we can choose the two
sections in (\ref{eqn:3.1}) to be different nonvanishing sections
and still get the same stability function $\psi$. Thus applying
stability theory, one has
   \begin{equation*}
   \int_{M_{st}} \langle \pi^* s_k, \pi^* s^{(k)}_{\Lambda_{red}} \rangle
   \mu \sim (\frac k{\pi})^{m/2}
   \int_{M_{red}} \langle s_k, s^{(k)}_{\Lambda_{red}} \rangle V
   \mu_{red}
   \end{equation*}
for all $s_k$ as $k \to \infty$. This together with
(\ref{eqn:5.15}) implies that asymptotically
   \begin{equation*}
   \pi^*s^{(k)}_{\Lambda_{red}} \sim (\frac k{\pi})^{m/2} \tilde
   s^{(k)}_\Lambda, \; k \to \infty.
   \end{equation*}
   %

%##############################
%        Section 6            %
%##############################

\section{Toric varieties}
\label{sec:6}

\subsection{The Delzant construction}
\label{sec:6.1}
Let $\LL = \CC^d \times \CC$ be the trivial line bundle over
$\CC^d$ equipped with the Hermitian inner product
   \begin{equation}
   \label{eqn:metric}
   \langle 1, 1 \rangle = e^{-|z|^2},
   \end{equation}
where $1: \CC^d \to \LL$, $z \mapsto (z,1)$ is the standard
trivialization of $\LL$. The line bundle $\LL$ is the
pre-quantum line bundle for $\CC^d$, since
   \begin{equation*}
   curv(\nabla) = -\sqrt{-1}\bar \partial \partial \log{\langle 1,
     1 \rangle} = \sqrt{-1} \sum d\bar z \wedge dz = -\omega.
   \end{equation*}

Let $K = (S^1)^d$ be the $d$-torus, which acts on $\CC^d$ by the
diagonal action,
   \begin{equation*}
   \tau(e^{it_1}, \cdots, e^{it_d}) \cdot (z_1, \cdots, z_d) =
   (e^{it_1}z_1, \cdots, e^{it_d}z_d).
   \end{equation*}
This is a Hamiltonian action with moment map
   \begin{equation}
   \label{eqn:6.1}
   \phi(z) = \sum_{i=1}^d |z_i|^2 e_i^*,
   \end{equation}
where $e_1^*, \cdots, e_d^*$ is the standard basis of $\fk^* =
\RR^d$.

Now suppose $G \subset K$ is an $m$-dimensional sub-torus of $K$,
$\fg=\,$Lie$(G)$ its Lie algebra, and $\ZZ_G^* \subset \fg^*$ the
weight lattice. Then the restriction of the $K$-action to $G$ is
still Hamiltonian, with moment map
   \begin{equation}
   \label{eqn:6.2}
   \Phi(z) = L \circ \phi(z) = \sum_{i=1}^d |z_i|^2 \alpha_i,
   \end{equation}
where $\alpha_i = L(e_i^*) \in \ZZ_G^*$, and $L: \mathfrak k^* \to
\fg^*$ is the transpose of the inclusion $\fg \hookrightarrow
\fk$.

We assume that the moment map $\Phi$ is proper, or alternatively,
that the $\alpha_i$'s are polarized: there exists $v \in \fg$ such
that $\alpha_i(v) > 0$ for all $1 \le i \le d$. Let $\alpha \in
\ZZ_G^*$ be fixed, with the property that $G$ acts freely on
$\Phi^{-1}(\alpha)$. Then the symplectic quotient at level
$\alpha$,
   \begin{equation*}
   M_{\alpha} = \Phi^{-1}(\alpha)/G,
   \end{equation*}
is a symplectic toric manifold; and by Delzant's theorem, all
toric manifolds arise this way.

The Hamiltonian action of $K$ on $\CC^d$ induces a Hamiltonian
action of $K$ on $M_\alpha$, with moment map $\Phi_\alpha$
defined by
   \begin{equation}
   \label{eqn:6.3}
   \phi \circ \iota_\alpha  = \Phi_\alpha \circ \pi_\alpha,
   \end{equation}
where $\iota_\alpha: \Phi^{-1}(\alpha) \hookrightarrow \CC^d$ is
the inclusion map, and $\pi_\alpha: \Phi^{-1}(\alpha) \to
M_\alpha$ the projection map. The moment polytope of this
Hamiltonian action on $M_\alpha$ is
   \begin{equation}
   \label{eqn:6.4}
   \Delta_\alpha = L^{-1}(\alpha) \cap \RR^d_+ =
   \{ t \in \RR^d \ |\ t_i \ge 0, \ \sum t_i \alpha_i = \alpha\}.
   \end{equation}
If we replace $\LL$ by $\LL^k$, i.e. the trivial line bundle over
$\CC^d$ with Hermitian inner product $\langle 1, 1
\rangle_k=e^{-k|z|^2}$, then everything proceeds as above, and the
moment polytope is changed to $k \Delta_\alpha =
\Delta_{k\alpha}$.

\subsection{Line bundles over toric varieties}
\label{sec:6.2}
As we showed in \S 2, $M_\alpha$ also admits the following GIT
description,
   \begin{equation*}
   M_\alpha = \CC_{st}^d(\alpha) / G_{\CC},
   \end{equation*}
where $G_\CC \simeq (\CC^*)^n$ is the complexification of $G$, and
$\CC_{st}^d(\alpha)$ is the $G_\CC$ flow-out of
$\Phi^{-1}(\alpha)$. This flow-out is easily seen to be identical
with the set
   \begin{equation}
   \label{eqn:6.5}
   \CC_{st}^d(\alpha) = \{z \in \CC^d\ |\ I_z \in
   I_{\Delta_\alpha}\},
   \end{equation}
where
   \begin{equation*}
   I_z = \{i\ |\ z_i \ne 0\}
   \end{equation*}
and
   \begin{equation*}
   I_{\Delta_\alpha} = \{I_t\ |\ t \in \Delta_\alpha\}.
   \end{equation*}

Now let $G$ act on the line bundle $\LL$ by acting on the trivial
section, $1$, of $\LL$, by weight $\alpha$. (In Kostant's formula
(\ref{eqn:2.7}) this has the effect of shifting the moment map
$\Phi$ by $\alpha$, so that the new moment map becomes $\Phi -
\alpha$ and the $\alpha$ level set of $\Phi$ becomes the zero
level set of $\Phi-\alpha$). This action extends to an action of
$G_\CC$ on $\LL$ which acts on the trivial section $1$ by the
complexification, $\alpha_\CC$, of the weight $\alpha$ and we can
form the quotient line bundle,
   \begin{equation*}
   \LL_\alpha = \iota_\alpha^* \LL/G = \LL_{st}(\alpha) / G_\CC,
   \end{equation*}
where $\LL_{st}(\alpha)$ is the restriction of $\LL$ to
$\CC_{st}^d(\alpha)$.

The holomorphic sections of $\LL^k_\alpha$ are closely related to
monomials in $\CC^d$. In fact, since $\LL$ is the trivial line
bundle, the monomials
   \begin{equation*}
   z^m = z_1^{m_1} \cdots z_d^{m_d}
   \end{equation*}
are holomorphic sections of $\LL$, and by Kostant's formula, $z^m$
is a $G$-invariant section of $\LL$ (with respect to the moment
map $\Phi_\alpha$) if and only if
   \begin{equation*}
   \tau^\#(\exp v)^* z^m = e^{i\alpha(v)} z^m
   \end{equation*}
for all $v \in \fg$; in other words, if and only if $m$ is an
integer point in $\Delta_\alpha$. So we obtain
   \begin{equation}
   \label{eqn:6.6}
   \Gamma_{hol}(\LL)^G = \span\{z^m\ |\ m \in \Delta_\alpha \cap
   \ZZ^d\}.
   \end{equation}
In view of (\ref{eqn:6.5}), $\CC_{st}^d(\alpha)$ is Zariski open,
so the GIT mapping
   \begin{equation*}
   \gamma: \Gamma_{hol}(\LL)^G \to \Gamma_{hol}(\LL_\alpha)
   \end{equation*}
is bijective, although $\CC^d$ is noncompact. As a result, the
sections
   \begin{equation}
   \label{eqn:6.7}
   s_m = \gamma(z^m), \;\ m \in \Delta_\alpha \cap \ZZ^d.
   \end{equation}
give a basis of $\Gamma_{hol}(\LL_\alpha)$.

To compute the norm of these sections $s_m$, we introduce the
following notation. Let $j: \Delta_\alpha \hookrightarrow \RR^d_+$
be the inclusion map, and $t_i$ the standard $i^{th}$ coordinate
functions of $\RR^d$. Then the \emph{lattice distance} of $x \in
\Delta_\alpha$ to the $i^{th}$ facet of $\Delta_\alpha$ is
$l_i(x)=j^* t_i(x)$. On $\Phi^{-1}(\alpha)$ one has
   \begin{equation*}
   \langle z^m, z^m \rangle = |z_1^{m_1}|^2 \cdots |z_d^{m_d}|^2
   e^{-|z|^2},
   \end{equation*}
which implies
   \begin{equation}
   \label{eqn:6.8}
   \langle s_m, s_m \rangle_\alpha = (\Phi_\alpha)^*
   (l_1^{m_1}\cdots
   l_d^{m_d}e^{-l})\ ,
   \end{equation}
where $l=l_1+\cdots+l_d$. As a corollary, we see that the
stability function on $\CC_{st}^d(\alpha)$ is
   \begin{equation}
   \label{eqn:6.9}
   \psi(z) = -|z|^2 + \log{|z^m|^2} - \pi^* \Phi_\alpha^*(\sum m_i
   \log{l_i} -l).
   \end{equation}
Finally by the Duistermaat-Heckman theorem the push-forward of the
symplectic measure on $M_\alpha$ by $\Phi_\alpha$ is the Lebesgue
measure $d\sigma$ on $\Delta_\alpha$, so the $L^2$ norm of $s_m$
is
   \begin{equation}
	\label{eqn:normform}
   \langle s_m, s_m \rangle_{L^2} = \int_{\Delta_\alpha}
   l_1^{m_1} \cdots l_d^{m_d} e^{-l}
   d\sigma.
   \end{equation}

For toric varieties, the ``Bohr-Sommerfeld" issues that we
discussed in \S 5.1 are easily dealt with: Let $\tilde s$ be the
$G$-invariant section, $z_1^{m_1}\cdots z_d^{m_d}$, of $\LL$, with
$(m_1, \cdots, m_d) \in \Delta_\alpha$. Then $\langle \tilde s,
\tilde s \rangle$ take its maximum on the set $\Phi^{-1}(m_1,
\cdots, m_d)$, and if $(m_1, \cdots, m_d)$ is in the interior of
$\Delta_\alpha$, this set is a Lagrangian torus: an orbit of
$\mathbb T^d$. Moreover, if $s$ is the section of $\LL_\alpha$
corresponding to $\tilde s$, $\langle s, s\rangle$ takes its
maximum on the projection of this orbit in $M_\alpha$, which is
also a Lagrangian submanifold.

\subsection{Canonical affines}
\label{sec:6.3}
We end this section by briefly describing a covering by natural
coordinate charts on $M_\alpha$ -- the canonical affines. (For
more details c.f. \cite{DuP}). Let $v$ be a vertex of $\Delta$.
Since $\Delta$ is a Delzant polytope, $\#I_v = n$ and
$\{\alpha_i\ |\ i \in I_v\}$ is a lattice basis of $\ZZ_G^*$.
Denote by
   \begin{equation}
   \label{eqn:6.10}
   \Delta_v = \{t \in \Delta\ |\ I_t \supset I_v\},
   \end{equation}
the open subset in $\Delta_\alpha$ obtained by deleting all facets
which don't contain $v$ and let
   \begin{equation*}
   Z_v = \Phi_\alpha^{-1}(\Delta_v).
   \end{equation*}
\begin{definition}
The \emph{canonical affines} in $M_{\alpha}$ are the open subsets
   \begin{equation}
   \label{eqn:6.11}
   \mathcal U_v = Z_v/G.
   \end{equation}
\end{definition}
Since $\{\alpha_i\ |\ i \in I_v\}$ is a lattice basis, for $j
\notin I_v$ we have $\alpha_j = \sum c_{j,i}\alpha_i$, where
$c_{j,i}$ are integers. Suppose $\alpha = \sum a_i \alpha_i$, then
$Z_v$ is defined by the equations
   \begin{equation}
   \label{eqn:6.12}
   |z_i|^2 = a_i - \sum_{j \not\in I_v} c_{j,i}|z_j|^2, \qquad i
   \in I_v
   \end{equation}
and the resulting inequalities
   \begin{equation}
   \label{eqn:6.13}
   \sum c_{j,i}|z_j|^2 < a_i.
   \end{equation}
So $\mathcal U_v$ can be identified with the set (\ref{eqn:6.12}).
The set
   \begin{equation*}
   z_i = \left( a_i - \sum_{j \notin I_v} c_{j,i}|z_j|^2 \right)^{1/2}
   \end{equation*}
is a cross-section of the $G$-action on $Z_v$, and the restriction
to this cross-section of the standard symplectic form on $\CC^d$
is $\sqrt{-1}\ \sum_{i \notin I_z} dz_i \wedge d\bar z_i$. So the
reduced symplectic form is
   \begin{equation}
   \omega_\alpha = \sqrt{-1}\ \sum_{j \notin I_v} dz_j \wedge
   d\bar z_j,
   \end{equation}
in other words, the $z_j$'s with $j \notin I_v$ are
\emph{Darboux coordinates} on $\mathcal U_v$.

%##############################
%        Section 7            %
%##############################

\section{Stability functions on toric varieties}
\label{sec:7}

\subsection{The general formula}
\label{sec:7.1}
In this section we compute the stability functions for the toric
varieties $M_\alpha$ defined above, with upstairs space $\CC^d$
and upstairs metric (\ref{eqn:metric}). For $z \in M_{st}$ there is a
unique $g \in \exp{\sqrt{-1}\fg}$ such that $g \cdot z \in
\Phi^{-1}(\alpha)$, and by definition, if $s(z)=z^m=\pi^* s_m$,
   \begin{equation}
   \label{eqn:7.1}
   \aligned
   \psi(z) &= \log{\langle s, s\rangle(z)} - \log{\langle s, s\rangle(g\cdot z)}
   \\
   &= -|z|^2+\log{|z^m|^2}+|g\cdot z|^2-\log{|(g\cdot z)^m|^2}.
   \endaligned
   \end{equation}
Moreover, If the circle group $(e^{i\theta}, \cdots, e^{i\theta})$
is contained in $G$, or alternatively, if $v=(1, \cdots, 1) \in
\fg$, or alternatively if $M_\alpha$ can be obtained by reduction
from $\CC\PP^{d-1}$, then
   \begin{equation*}
   |z|^2=\sum \alpha_i(v)|z_i|^2=\langle \Phi(z), v\rangle,
   \end{equation*}
thus
   \begin{equation}
   \label{eqn:7.2}
   |g \cdot z|^2 = \langle \Phi(g\cdot z), v\rangle = \langle
   \alpha, v\rangle,
   \end{equation}
and (\ref{eqn:7.1}) simplifies to
   \begin{equation}
   \label{eqn:7.3}
   \psi(z) = -|z|^2 + \log|z^m|^2 + \alpha(v) - \log|(g\cdot
   z)^m|^2.
   \end{equation}

Given a weight $\beta \in \ZZ_G^*$ let $\chi_\beta: G_\CC \to \CC$
be the character of $G_\CC$ associated to $\beta$. Restricted to
$\exp(\sqrt{-1}\fg)$, $\chi_\beta$ is the map
   \begin{equation}
   \label{eqn:7.4}
   \chi_\beta(\exp{i\xi}) = e^{-\beta(\xi)}.
   \end{equation}
Now note that by (\ref{eqn:7.3}),
   \begin{equation*}
   \aligned
   \psi(z) &= -|z|^2 + \alpha(v) + \log|z^m|^2 -
   \log(\prod \chi_{\alpha_i}(g)^{2 m_i}) |z^m|^2 \\
   & = -|z|^2+\alpha(v) - \log\prod \chi_{\alpha_i(g)^{2m_i}}.
   \endaligned
   \end{equation*}
But $z^m=\pi^* s_m$ for $s_m \in \Gamma_{hol}(\LL_\alpha)$ if and
only if $m$ is in $\Delta_\alpha$, i.e. $\sum m_i \alpha_i =
\alpha$, so we get finally by (\ref{eqn:7.4}), $\prod
\chi_{\alpha_i}(g)^{m_i} = \chi_\alpha (g)$ and
   \begin{equation}
   \label{eqn:7.5}
   \psi(z) = -|z|^2 + \alpha(v) - 2 \log \chi_\alpha(g).
   \end{equation}
Recall now that the map
   \begin{equation*}
   \Phi^{-1}(\alpha) \times \exp(\sqrt{-1}\fg) \to \CC^d_{st}
   \end{equation*}
is bijective, so the inverse of this map followed by projection
onto $\exp(\sqrt{-1}\fg)$ gives us a map
   \begin{equation}
   \label{eqn:7.6}
   \gamma: \CC^d_{st} \to \exp(\sqrt{-1}\fg),
   \end{equation}
and by the computation above we've proved
\begin{theorem}
\label{thm:7.1} The stability function for $M_\alpha$, viewed as a
GIT quotient of $\CC^d$ with the trivial line bundle and the flat
metric (\ref{eqn:metric}), is
   \begin{equation}
   \label{eqn:7.7}
   \psi(z) = -|z|^2 + \alpha(v) - 2(\log \gamma^*\chi_\alpha)(z).
   \end{equation}
\end{theorem}

For example for $\CC\PP^{n-1}$ itself with $\CC^n_{st}=\CC^n -
\{0\}$ and $\alpha = 1$, $\gamma(z)=|z|$ and hence
   \begin{equation}
   \label{eqn:7.8}
   \psi(z) = -|z|^2 + 1 + \log{|z|^2}.
   \end{equation}

The formula (\ref{eqn:7.7}) is valid modulo the assumption that
$M_\alpha$ can be obtained by reduction from $\CC\PP^{d-1}$, i.e.
modulo the assumption (\ref{eqn:7.2}). Dropping this assumption we
have to replace (\ref{eqn:7.7}) by the slightly more complicated
formula
   \begin{equation}
   \label{eqn:7.9}
   \psi(z) = -|z|^2 + |\gamma(z)^{-1}z|^2 -
   2(\log\gamma^*\chi_\alpha)(z).
   \end{equation}

\subsection{Stability function on canonical affines}
\label{sec:7.2}
We can make the formula (\ref{eqn:7.7}) more explicit by
restricting to the canonical affines, $\mathcal U_v$, of \S
\ref{sec:6.3}. For any vertex $v$ of $\Delta$ it is easy to see
that
   \begin{equation*}
   \mathcal U_v = \CC^d_{\Delta_v} / G_\CC,
   \end{equation*}
where
   \begin{equation}
   \label{eqn:7.10}
   \CC^d_{\Delta_v} = \{z \in \CC^d\ |\ I_z \supset I_v\}
   \end{equation}
is an open subset of $\CC^d_{st}$. By relabelling we may assume
$I_v = \{1, 2 \cdots, n\}$. Since the relabelling makes $\alpha_1,
\cdots, \alpha_n \in \mathbb \fg^*$ into a lattice basis of
$\ZZ_G^*$, $\alpha_k = \sum c_{k,i}\alpha_i$ for $k > n$, where
$c_{k,i}$ are integers. Let $f_1, \cdots, f_n$ be the dual basis
of the group lattice, $\ZZ_G$, then the map
   \begin{equation}
   \label{eqn:7.11}
   \CC^n \to G_\CC, \quad (w_1, \cdots, w_n) \mapsto
   w_1 f_1 + \cdots + w_n f_n \mod \ZZ_G
   \end{equation}
gives one an isomorphism of $G_\CC$ with the complex torus
$(\CC^*)^n$ and in terms of this isomorphism the $G_\CC$-action
on $\CC^d_{\Delta_v}$ is given by
   \begin{equation*}
   (w_1, \cdots, w_n) \cdot z = \left(w_1 z_1, \cdots, w_n z_n,
   (\prod_{i=1}^n w_i^{c_{n+1, i}})z_{n+1}, \cdots,
   (\prod_{i=1}^n w_i^{c_{d, i}})z_{d}\right).
   \end{equation*}
Now suppose $z \in \CC^d_{\Delta_v}$. Then the system of equations
obtained from (\ref{eqn:6.12}) and (\ref{eqn:7.1}),
   \begin{equation*}
   r_i^2|z_i|^2 + \sum_{k=n+1}^d c_{k,i} (\prod_{j=1}^n
   r_j^{c_{k,i}})^2 |z_k|^2  = a_i, \quad 1 \le i \le n,
   \end{equation*}
has a unique solution, $g=(r_1(z), \cdots, r_n(z)) \in (\RR^+)^n =
\exp{(\sqrt{-1}\fg)}$, i.e.,  the $g$ in (\ref{eqn:7.1}) is $(r_1,
\cdots, r_n)$. Via the identification (\ref{eqn:7.10}) the weight
$\alpha \in \ZZ_G^*$ corresponds by (\ref{eqn:7.11}) to the weight
$(a_1, \cdots, a_n) \in \ZZ^n$ and by (\ref{eqn:7.7}) and
(\ref{eqn:7.9})
   \begin{equation}
   \label{eqn:7.12}
   \psi|_{\CC^d_{\Delta_v}}  = -|z|^2 + \alpha(v) - 2 \sum_i
   a_i \log r_i(z)
   \end{equation}
in the projective case and
   \begin{equation}
   \label{eqn:7.13}
   \psi|_{\CC^d_{\Delta_v}}  = -|z|^2 + \sum_i r_i^2|z_i|^2 +
   \sum_{k>n}|(\prod_{i=1}^n r_i^{c_{k,i}})z_k|^2 - 2\sum_i
   a_i \log{r_i}.
   \end{equation}
in general.

\subsection{Example: The stability function on the Hirzebruch
surfaces}
\label{sec:7.3}
As an example, let's compute the stability function for Hirzebruch
surfaces. Recall that the Hirzebruch surface $H_n$ is the toric
4-manifold whose moment polytope is the polygon with vertices $(0,
0), (0, 1), (1, 1), (n+1, 0)$. By the Delzant construction, we see
that $H_n$ is in fact the toric manifold obtained from the
$\TT^2$-action on $\CC^4$,
   %^
   \begin{equation*}
   (e^{i\theta_1}, e^{i\theta_2}) \cdot z = (e^{i\theta_1}z_1,
   e^{i\theta_2}z_2,
   e^{i\theta_1 - i n \theta_2}z_3, e^{i\theta_2} z_4).
   \end{equation*}
By the procedure above, we find the stability function
   \begin{equation*}
   \psi(z) = -|z|^2 - a_1 \log{r_1} -a_2\log{r_2} +a_1+a_2-n
   r_1^{2n} r_2^{2}|z_3|^2,
   \end{equation*}
where $r_1, r_2$ are the solution to the system of equations
   \begin{equation*}
   \aligned
   r_1^2|z_1|^2 &+ r_1^{2n}r_2^2 |z_3|^2 = a_1,\\
   r_2^2|z_2|^2 &- nr_1^{2n}r_2^2 |z_3|^2 + r_1^2 |z_4|^2 = a_2.
   \endaligned
   \end{equation*}
   %

%##############################
%        Section 8            %
%##############################

\section{Applications of stability theory to toric varieties}
\label{sec:8}

\subsection{Universal rescaled law on toric varieties}
\label{sec:8.1}
In this section we suppose $\beta \in \Delta_\alpha$ is rational,
and $N$ is large with $N\beta \in \ZZ^d$. One of the main results
in \cite{STZ} is the following universal rescaled law for the
probability distribution function (\ref{eqn:5.6}) on toric
varieties with respect to any Bergman toral metric,
   \begin{equation}
   \label{eqn:8.1}
   \lim_{N \to \infty} (\frac N{\pi})^{n/2}
   \sigma_{N,N\beta}((\frac
   N{\pi})^{n/2}t)= \frac{(\log c/t)^{n/2}}{c\Gamma(n/2 + 1)}.
   \end{equation}
By measure theoretic arguments, they deduce this
from moment estimates, (c.f. \S 4.1 of \cite{STZ})
   \begin{equation}
   \label{eqn:8.2}
   \int_{M_\alpha} x^l d\nu_{N} \to
   \frac{c^{l-1}}{l^{n/2}}, \indent N \to \infty,
   \end{equation}
where $l$ is any positive integer, $\nu_N$ is the push-forward
measure
   \begin{equation*}
   \nu_{N} = \left(\left|(\frac N{\pi})^{-n/4}\phi_{N\beta}
   \right|^2\right)_*
   \left((\frac N{\pi})^{n/2} \nu\right),
   \end{equation*}
with $\phi_{N\beta}=s_{N\beta}/\|s_{N\beta}\|$ and $\nu$ the
pullback of the Fubini-Study form via a projective embedding. By a
simple computation it is easy to see that
   \begin{equation}
   \label{eqn:8.3}
   \int x^l d\nu_{N}(x) = \left(\frac
   N{\pi}\right)^{-\frac{n(l-1)}2}
   \int_{M_\alpha}
   |\phi_{N\beta}|^{2l} \nu
   =\left(\frac N{\pi}\right)^{-\frac{n(l-1)}2}
   m_\alpha(l, \phi_{N\beta}, \nu).
   \end{equation}

Instead of considering the pullback of the Fubini-Study measure we
will consider another natural measure on toric varieties: the
quotient measure induced by the upstairs flat metric. The upstairs
analogue of (\ref{eqn:8.1}) for toric varieties is rather easy to
prove:
\begin{lemma}
\label{lemma:8.1}
For any $l$, the $l^{th}$ moments
   \begin{equation}
   \label{eqn:8.4}
   \left(\frac N{\pi}\right)^{-d(l-1)/2}
   m(l, \frac{z^{N\beta}}{\|z^{N\beta}\|}, d\mu) \to
   \frac{c^{l-1}}{l^{d/2}} \qquad (N \to \infty).
   \end{equation}
\end{lemma}
\begin{proof}
See \cite{GuW}, or by direct computation.
\end{proof}

Thus we can apply proposition \ref{prop:5.4} to derive
(\ref{eqn:8.2}) from (\ref{eqn:8.4}). By (\ref{eqn:5.7}) and
(\ref{eqn:4.8}),
   \begin{equation*}
   m_{\alpha}(l, \frac{s_{N\beta}}{\|s_{N\beta}\|},
   \frac{\omega_{\alpha}^n}{n!})
   \sim l^{-m/2} \left(\frac N{\pi}\right)^{m(l-1)/2} m(l,
   \frac{z^{N\beta}}{\|z^{N\beta}\|}, \frac{\omega^d}{d!}).
   \end{equation*}
Thus
   \begin{equation*}
   \left( \frac N{\pi}\right)^{-n(l-1)/2} m_\alpha(l, \frac
    {s_{N\beta}}{\|s_{N\beta}\|}, \frac{\omega_\alpha^n}{n!})
   \to \frac{c^{l-1}}{l^{n/2}}
   \end{equation*}
as $N \to \infty$ for all $l$. This together with the measure
theoretic arguments alluded to above implies the distribution law
(\ref{eqn:8.1}) for the special volume form $V\mu_\alpha$ on $M_\alpha$ associated to the metric (\ref{eqn:metric}).

\begin{remark}
Here we only consider the case when $\beta$ is an interior point
of the Delzant polytope, which corresponds to the case $r=0$ in
\cite{STZ}. However, one can modify the arguments above slightly
to show the same result for general $r$ and $N\beta$ replaced by
$N\beta + o(1)$.
\end{remark}

\subsection{Spectral properties of toric varieties}
\label{sec:8.2}
As we have seen, the stability theory derived in \S1-\S5 is
particularly useful for toric varieties $M_\alpha$, since the
upstairs space is the complex space, $\CC^d$, the Lie group $G$ is
abelian, its action on $\CC^d$ is linear, and the $G$-invariant
sections of $\LL$ are just linear combinations of monomials. As a
consequence, the expressions (\ref{eqn:1.9}), (\ref{eqn:1.13}),
(\ref{eqn:1.18}) etc. are relatively easy to compute.

For example, consider the density of states
   \begin{equation*}
   \mu_N = \sum \langle s_{N,i}, s_{N,i}\rangle \mu_{red},
   \end{equation*}
where $\{s_{N,i}\}$ is an orthonormal basis of $\Gamma_{hol}(\LL^N)$,
then by proposition \ref{prop:5.3},
   \begin{equation*}
   \int_{M_\alpha} f\mu_N \sim \left(\frac N{\pi}\right)^{-m/2}
   \int_{\CC^d} \frac{\pi^*f}{\pi^*V}\chi \sum\langle \pi^* s_{N,i},
   \pi^* s_{N,i}\rangle \mu.
   \end{equation*}
The right hand side has a very simple asymptotic expansion in
terms of Stirling numbers of the first kind (See \cite{GuW}, \cite{Wan}) and
from this and the results of \S \ref{sec:5.4} one gets an
alternative proof of theorem 1.1 of \cite{BGU}:
\begin{theorem}
\label{thm:8.3} There exists differential operators $P_i(x,D)$ of
order $2i$ such that
   \begin{equation*}
   \mu_N(f) \sim \sum_i N^{d-m-i} \int P_i(x,D)f(x)dx, \quad N \to \infty.
   \end{equation*}
\end{theorem}
In this way the coefficients of the downstairs density of states
asymptotics can be computed explicitly by the coefficients of the
asymptotic expansion of the invariant upstairs density of states
asymptotics -- the relation of the leading terms is given in
theorem \ref{thm:5.5}, and the other coefficients depend on the
asymptotics of the Laplace integral (\ref{eqn:4.1}) together with
the value of the stability function near $\Phi^{-1}(0)$. Similarly
theorem 1.2 of \cite{BGU} can be derived from the results of \S
\ref{sec:5.2} and upstairs analogues of these results in
\cite{GuW}.

\subsection{General toric K\"ahler metrics on the open stratum of a toric variety}
\label{sec:8.3}

In this section we will briefly indicate how the results of \S 8.1 and \S 8.2 can be extended to a much larger class of metrics on the toric variety above. We will confine ourselves, however, to describing these results on the open $K_{\CC}$ stratum of the toric variety and we hope to deal elsewhere with the implicit boundary conditions needed to describe those metrics on the open $K_{\CC}$ which would extend across the lower dimensional strata in the toric variety. In the GIT description of the toric variety above, 
$$M = (\CC^d)_{st}/G_{\CC},$$
the open stratum is the image in $\CC^d_{st}/G_{\CC}$ of the complex torus $(\CC^{*})^d$. Now let $\omega$ be any $K$-invariant K\"ahler form on $(\CC^{*})^d$. Then if the action of $K$ on $(\CC^{*})^d$ is Hamiltonian, there exists a $K$-invariant potential function, say $F$, for $\omega$, i.e.,
\begin{equation}
\label{eqn:8.5}
\omega = 2 \sqrt{-1} \partial \bar{\partial} F,
\end{equation}
and via the identification
$$(\CC^{*})^d = \CC^d/ 2\pi\sqrt{-1} \ZZ^d,$$
$F$ satisfies, by $K$-invariance, 
$$F(x + \sqrt{-1}y) = F(x),$$
and by the K\"ahler condition, $F$ is strictly convex as a function of $x$. Moreover, the moment map associated with the action of $K$ on $\CC^d/2\pi\sqrt{-1}\ZZ^d$ is just the Legendre transform, $\frac{\partial F}{\partial x}$. (For proofs of these assertions, see \cite{Gui94a}, \S 4.) Thus if $L$ is as in (\ref{eqn:6.2}) the dual $L: \mathfrak{k}^{*} \rightarrow \mathfrak{g}^{*}$ of the inclusion of $\mathfrak{g}$ into $\mathfrak{k}$, the moment map associated with the action of $G$ on $\CC^d/2\pi\sqrt{-1}\ZZ^d$ is $L\circ\frac{\partial F}{\partial x}.$ 

Let's now consider, as in (\ref{eqn:metric}), the trivial line bundle $\LL = \CC^d \times \CC$ over $\CC^d$ and lets $G$ act on this bundle by acting on $\CC$ by the weight $\alpha \in \ZZ_{G}$. Let us replace, however, the Hermitian inner product in (\ref{eqn:metric}) by the inner product
\begin{equation}
\label{eqn:metric_new}
\langle 1, 1\rangle = e^{-2 F}.
\end{equation}
Then for $k \in \ZZ^d_{+}$ with $L(k) = \alpha$, and $z = e^{x + \sqrt{-1}y}$,
$$\langle z^k, z^k \rangle = e^{2(k\cdot x - F)}.$$
Recall now that $(\frac{\partial F}{\partial x})^{-1}$ is also a Legendre transform, i.e., 
$$\frac{\partial F}{\partial x}(x) = t \iff x = \frac{\partial G}{\partial t}(t),$$
where 
\begin{equation}
\label{eqn:InverseLegTransform}
G(t) = x \cdot t - F(x).
\end{equation}
Thus 
\begin{equation}
\label{eqn:ILTrelation}
((\frac{\partial F}{\partial x})^{-1})^{*} \exp 2(k\cdot x - F(x)) = \exp 2(k\cdot \frac{\partial G}{\partial t} - t \cdot \frac{\partial G}{\partial t} + G(t))
\end{equation}
and its restriction to $L^{-1}(\alpha)$ with $t_i|_{L^{-1}(\alpha)} = \ell_i$ is 
\begin{equation}
\label{eqn:rho_k}
\rho_k = \exp 2 (\sum (\ell_i(k) - \ell_i) \frac{\partial G}{\partial t_i}(\ell_1,\ldots, \ell_d) + G(\ell_1,\ldots, \ell_d)).
\end{equation}
In particular, the function
\begin{equation}
\label{eqn:potfngen}
\phi^{*}\Log \rho_k
\end{equation}
is a potential function for the reduced K\"ahler metric on the image of $(\CC^{*})^d$ in $\CC^d_{st}/G_{\CC}$. Notice in particular that for the flat K\"ahler metric on $\CC^d$, $F(x) := F_0(x) = e^{x_1} + \ldots + e^{x_d}, \frac{\partial F_0}{\partial x_i} = e^{x_i}, \frac{\partial G_0}{\partial t_i} = \, \Log \, t_i,$ and 
$$G_0(t_1, \ldots , t_d) = \sum x_i e^{x_i} - F(x) = \sum t_i \, \Log \, t_i - t_i,$$
so that 
\begin{equation}
\label{eqn:logrho}
\frac 12 \, \Log \, \rho_k = \sum (\ell_i(k) - \ell_i) \, \Log \, \ell_i + \ell_i \, \Log \, \ell_i - \ell_i = \sum \ell_i(k)\, \Log \, \ell_i - \ell_i
\end{equation}
(compare with equations (\ref{eqn:6.8}) through (\ref{eqn:normform})). The discrepancy of ``$\frac 12$" is accounted for by the presence of the ``2" in (\ref{eqn:8.5}).

We now note that the stability function on $\CC^d/2\pi\sqrt{-1}\ZZ^d$ associated with the K\"ahler form is
$$F(x) - (\pi^{*}\circ\phi^{*} \rho_k)(x +\sqrt{-1} y),$$
or, by (\ref{eqn:7.6}), 
\begin{equation}
\label{eqn:newpot}
F(x) - \Phi^{*}\rho_k(\gamma^{-1}(x))
\end{equation}
and that given this expansion for the potential function one can generalize the results of \S 8.1 - \S 8.2 on probability distribution functions and spectral measures to these more general reduced K\"ahler metrics by means of the techniques described in \S 5.

In the discussion above we have confined ourselves to reduced metrics on the open stratum in $\CC^d_{st}/G_{\CC}$. We note in passing, however, that if we add small $K$-invariant perturbations $P$ to the euclidean $F_0$ above, such that the support of $P$ is compact in $(\CC^{*})^d$, then $F := F_0 + P$ gives rise to a reduced metric on all of $\CC^d_{st}/G_{\CC}$ for which the results of  \S 6 and \S 7 above are true, by the same arguments, using the calculations above. We will discuss elsewhere the boundary assumptions necessary on more general $F$ in order for these results to extend to the whole of $\CC^d_{st}/G_{\CC}$. 

As noted already towards the end of the introduction to this paper, asymptotic results as in \S 6-8 for broad classes of toric K\"ahler metrics on $\CC^d_{st}/G_{\CC}$ have been discussed (explicitly) in \cite{STZ} and (implicitly) in \cite{SoZ}.

%##############################
%        Section 9            %
%##############################

\section{Martin's construction}
\label{sec:9}

The non-abelian generalizations of toric varieties are
``spherical" varieties, and the simplest examples of these are
coadjoint orbits and varieties obtained from coadjoint orbits by
symplectic cuts. In the remainder of this paper we apply stability
theory to the coadjoint orbits of the unitary group $\U(n)$. It is
well known that the coadjoint orbits of $\U(n)$ can be identified
with the sets of isospectral Hermitian matrices $\mathcal
H(\lambda) \subset \mathcal H(n)$, i.e.,  Hermitian matrices with
fixed eigenvalues $\lambda_1 \ge \lambda_2 \ge \cdots \ge
\lambda_n$. For $\lambda_1 = \cdots = \lambda_{n-1} > \lambda_n$,
$\mathcal H(\lambda)$ is $\CC\PP^{n-1}$ which is a toric manifold.
Thus the first non-toric case is given by $\lambda_1 = \cdots =
\lambda_k > \lambda_{k+1} = \cdots = \lambda_n$, $1 < k < n-1$, in
which case $\mathcal H(\lambda)$ is the complex Grassmannian
$Gr(k, \CC^n)$.

\subsection{GIT for Grassmannians}
\label{sec:9.1}
Suppose $k < n$. It is well known that the complex Grassmannian
$Gr(k,\CC^n)$ can be realized as the quotient space of $\CC^{kn}$
by symplectic reduction or as a GIT quotient as follows:

Let $M=\mathfrak M_{k,n}(\CC) \simeq \CC^{kn}$ be the space of
complex $k \times n$ matrices. We equip $\CC^{kn}$ with its
standard K\"ahler metric, the standard trivial line bundle $\CC
\times \CC^{kn} \to \CC^{kn}$, and the standard Hermitian inner
product on this line bundle,
   \begin{equation}
   \label{eqn:9.1}
   \langle 1, 1 \rangle(Z) = e^{-\tr ZZ^*}.
   \end{equation}

Now let $G=\U(k)$ act on $\mathfrak M_{k,n}$ by left
multiplication. This action preserves the inner product
(\ref{eqn:9.1}), and thus preserves the K\"ahler form
$\sqrt{-1}\partial \bar
\partial \tr ZZ^*.$ It is not hard to see that it is a
Hamiltonian action with moment map
   \begin{equation}
   \label{eqn:9.2}
   \Phi:  \mathfrak M_{k,n} \to \mathcal H_k, \qquad Z \mapsto
   ZZ^*,
   \end{equation}
where $\mathcal H_k$ is the space of $k \times k$ Hermitian
matrices. Here we identify $\mathcal H_k$ with $\sqrt{-1}\mathcal
H_k =Lie(\U(k))$, and identify $\mathcal H_k$ with
$Lie(\U(k))^*=\mathcal H_k^*$ via the Killing form. Notice that
the identity matrix $I$ lies in the annihilator of the commutator
ideal,
   \begin{equation*}
   [\mathcal H_k, \mathcal H_k]^0 = \{a \in \mathcal H_k^*\ |\ \langle
   [h_1, h_2], a \rangle=0 \mbox{\ for\ all\ } h_1, h_2 \in
   \mathcal H_k\},
   \end{equation*}
so $\Phi-I$ is also a moment map, and it's clear that the
reduced space
   \begin{equation*}
   M_{red} = \Phi^{-1}(I)/G
   \end{equation*}
is the Grassmannian $Gr(k, \CC^n)$.

On the other hand, the complexification of $\U(k)$ is $GL(k,
\CC)$, and it's not hard to see that the set of stable points,
$M_{st}$, is exactly the set of $k \times n$ matrices $A \in M$
which have rank $k$, and that the quotient $M_{st}/GL(k, \CC)$ is
again $Gr(k, \CC^n)$ . This gives us the GIT description of $Gr(k,
\CC^n)$.

As for the reduced line bundle, $\LL_{red}$, on $M_{red}$, this is
obtained from the trivial line bundle on $M_{st}$ by ``shifting"
the action of $GL(k,\CC)$ on the trivial line bundle in conformity
with the shifting, ``$\Phi \Rightarrow \Phi-I$", of the moment
map, i.e. by letting $GL(k, \CC)$ act on this bundle by the
character
   \begin{equation*}
   \gamma: GL(k,\CC) \to \CC^*, \gamma(A) = \det(A).
   \end{equation*}

\subsection{Martin's reduction procedure}
\label{sec:9.2}
For general coadjoint orbit of $\U(n)$, Shaun Martin showed that
there is an analogous GIT description. Since he never published
this result, we will roughly outline his argument here, focusing
for simplicity on the case $\lambda_1 > \cdots > \lambda_n$.

Let
   \begin{equation*}
   M = \mathfrak M_{1,2}(\CC) \times \mathfrak M_{2, 3}(\CC)
   \times \cdots \times\mathfrak M_{n-1, n}(\CC).
   \end{equation*}
Then each component of $M$ is a linear symplectic space, and $M$
is just the linear symplectic space $ \CC^{(n-1)n(n+1)/3}$ \ with
standard K\"ahler form $\omega=-\sqrt{-1}\partial \bar \partial
\log \rho$, where $\rho$ is the potential function
   \begin{equation*}
   \rho(Z) = \exp(-\sum_{i=1}^{n-1}\tr Z_i Z_i^*).
   \end{equation*}
Consider the group
   \begin{equation*}
   G=\U(1) \times \U(2) \times \cdots \times \U(n-1)
   \end{equation*}
acting on $M$ by the recipe:
   \begin{equation}
   \label{eqn:9.3}
   \tau_{(U_1, \cdots, U_{n-1})} (Z_1,  \cdots, Z_{n-1}) =
   (U_1 Z_1 U_2^*,  \cdots, U_{n-2}Z_{n-2}U_{n-1}^*,
   U_{n-1}Z_{n-1}).
   \end{equation}
\begin{lemma}
\label{lemma:9.1}
The action above is Hamiltonian with moment map
   \begin{equation}
   \label{eqn:9.4}
   \Phi(Z_1, \cdots, Z_{n-1}) = (Z_1 Z_1^*, Z_2 Z_2^* - Z_1^*Z_1,
   \cdots, Z_{n-1}Z_{n-1}^* - Z_{n-2}^* Z_{n-2}).
   \end{equation}
\end{lemma}
\begin{proof}
Given any $H=(H_1, \cdots, H_{n-1}) \in \mathcal H_1 \times
\cdots \times \mathcal H_{n-1}$, denote by $\mathcal U_H(t)$ the
one parameter subgroup of $G$ generated by $H$, i.e.,
   \begin{equation*}
   \aligned
   \mathcal U_H(t) Z =& \left(\exp{(\sqrt{-1}tH_{1})}Z_{1}
   \exp{(-\sqrt{-1}tH_{2})}, \cdots,\right. \\
   & \
   \left.\exp{(\sqrt{-1}tH_{n-2})}Z_{n-2}\exp{(-\sqrt{-1}tH_{n-1})},
   \exp{(\sqrt{-1}tH_{n-1})}Z_{n-1}\right).
   \endaligned
   \end{equation*}
Let $v_H$ be the infinitesimal generator of this group, then
   \begin{equation*}
   \iota_{v_H}(\sqrt{-1}\partial \log \rho) = -\sqrt{-1} \sum
   \tr((\iota_{v_H}dZ_i)Z_i^*).
   \end{equation*}
Since
   \begin{equation*}
   \iota_{v_H}dZ_i %= L_{v_H}Z_i
   = \left.\frac d{dt}(\exp{(\sqrt{-1}tH_{i})}
    Z_{i}\exp{(-\sqrt{-1}tH_{i+1})})\right|_{t=0}
   =\sqrt{-1}(H_i Z_i -Z_i H_{i+1}),
   \end{equation*}
we see that
   \begin{equation*}
   \iota_{v_H}(\sqrt{-1}\partial \log \rho) = \sum \tr(H_i
   Z_iZ_i^* - H_{i+1}Z_i^*Z_i) = \langle \Phi(Z), H \rangle.
   \end{equation*}
This shows that (\ref{eqn:9.4}) is a moment map of $\tau$.
\end{proof}

Given $a=(a_1, \cdots, a_n) \in \RR_+^n$, let
   \begin{equation*}
   \phi^{-1}(aI) = \Phi^{-1}(a_1 I_1, \cdots, a_{n-1}I_{n-1}),
   \end{equation*}
and let
   \begin{equation*}
   M_a =\Phi^{-1}(aI)/G
   \end{equation*}
be the reduced space at level $(a_1 I_1, \cdots, a_{n-1}I_{n-1})
\in [\fg, \fg]^0$. Consider the residual action of $GL(n, \CC)$ on
$M$,
   \begin{equation}
   \label{eqn:9.5}
   \kappa: GL(n, \CC)\times M \to M, \quad \kappa_A Z =
   (Z_1, \cdots, Z_{n-2}, Z_{n-1}A^{-1}).
   \end{equation}
Then the actions $\kappa$ and $\tau$ commute, and by the same
argument as above we see that $\kappa|_{\U(n)}$ is a Hamiltonian
action with a moment map
   \begin{equation}
   \label{eqn:9.6}
   \Psi: M \to \mathcal H_n, \quad \Psi(Z) = Z_{n-1}^* Z_{n-1} +
   a_n I_n.
   \end{equation}
We thus get a Hamiltonian action of $\U(n)$ on the reduced space
$M_a$ with moment map $\Psi_a: M_a \to \mathcal H_n$, which
satisfies $\Psi \circ i = \Psi_a \circ \pi_0$, where, as usual,
$i: \Phi^{-1}(aI) \hookrightarrow M$ is the inclusion map and
$\pi_0: \Phi^{-1}(aI) \to M_a$ the projection.

\begin{theorem}[\cite{Mar}]
\label{thm:9.2}
$\Psi_a$ is a $\U(n)$-equivariant
symplectomorphism of $M_a$ onto $\mathcal H(\lambda)$, with
$\lambda_i = \sum_{j=i}^n a_j$.
\end{theorem}
\begin{proof}
First we prove that $\Psi_a$ maps $M_a$ onto the isospectral set
$\mathcal H(\lambda)$. In view of the relation $\Psi \circ i =
\Psi_a \circ \pi_0$, we only need to show Image$(\Psi) =\mathcal
H(\lambda)$. In fact, if $Z_i Z_i^*$ has eigenvalues $(\mu_1,
\cdots, \mu_i)$, then the eigenvalues of $Z_iZ_i^*$ are exactly
$(\mu_1, \cdots, \mu_i, 0)$, so it is straightforward to see that
$Z_2Z_2^*=Z_1Z_1^*+a_2I_2$ has eigenvalues $a_1+a_2, a_2$, and in
general $Z_i Z_i^*$ has eigenvalues $a_1+\cdots+a_i,
a_2+\cdots+a_i, \cdots, a_i$. This proves that $\Psi_a$ maps $M_a$
into $\mathcal H(\lambda)$, and since $G$ acts transitively on
$\mathcal H(\lambda)$, this map is onto.

Next note that by dimension-counting $\dim M_a = \dim \mathcal
H(\lambda)$, so $\Psi_a$ is a finite-to-one covering. Since the
adjoint orbits of $\U(n)$ are simply-connected, we conclude that
this map is also injective, and thus a diffeomorphism.

Since $\Psi_a$ is a moment map, it is a Poisson mapping between
$M_a$ and $\mathcal H(n)$, i.e.,
   \begin{equation*}
   \{ f \circ \Psi_a, g \circ \Psi_a \}_{M_a} = \{f, g\}_{\mathcal
   H(\lambda)} \circ \Psi_a
   \end{equation*}
for any $f, g \in C^\infty(\mathcal H(\lambda))$. Thus $\Psi_a$ is
a symplectomorphism between $M_a$ and $ \mathcal H(\lambda)$.

Finally the $\U(n)$-equivariance comes from the fact that
   \begin{equation*}
   \Psi(U \cdot Z) = (U^{-1})^* Z_{n-1}^* Z_{n-1}U^{-1}+a_n I_n =
   U (Z_{n-1}^*Z_{n-1}+a_nI_n) U^{-1}
   =U \cdot \Psi(Z).
   \end{equation*}
This completes the proof.
\end{proof}

The GIT description of this reduction procedure is now clear:
$$Z=(Z_1, \cdots, Z_{n-1}) \in M_{st}$$ if and only if $Z_i$ is of
rank $i$ for all $i$, and
   \begin{equation*}
   M_a  = M_{st}/G_\CC
   \end{equation*}
with $G_\CC$  the product
   \begin{equation*}
   G_\CC = GL(1, \CC)\times \cdots \times GL(n-1, \CC).
   \end{equation*}

\subsection{Twisted line bundles over $\U(n)-$coadjoint orbits}
\label{sec:9.3}
As in the toric case, reduction at level 0 of the moment map
(\ref{eqn:9.2}) is not very interesting, since the reduced line
bundle is the trivial line bundle. To get the Grassmannian, we
shifted the moment map by the identity matrix. Equivalently, we
``twisted'' the action of $GL(k, \CC)$ on the trivial line bundle
$\CC \times \CC^{kn}$ by a character of $GL(k, \CC)$. It is to
this shifted moment map/twisted action that we applied the
reduction procedure to obtain a reduced line bundle on $Gr(k,
\CC^n)$.

Similarly, for $\U(n)$-coadjoint orbits we will twist the $G_\CC$
action on the trivial line bundle over $M$ by characters of
$G_\CC$. Every character of $G_\CC$ is of the form
   \begin{equation}
   \label{eqn:9.7}
   \gamma = \gamma_1^{m_1} \cdots \gamma_{n-1}^{m_{n-1}},
   \end{equation}
where $\gamma_k(A) = \det (A_k)$ for $A=(A_1, \cdots, A_{n-1})$.
Let
   \begin{equation*}
   \pi_k : M \to \mathfrak M_{k,n}, \quad (Z_1, \cdots, Z_{n-1})
   \to Z_k Z_{k+1} \cdots Z_{n-1}.
   \end{equation*}
Then $\pi_k$ intertwines the action of $G_\CC$ on $M$ with the
standard left action of $\mathcal U(k)$ on $\mathfrak M_{k,n}$,
and intertwines the action $\kappa$ of $\mathcal U(n)$ on $M$ with
the standard right action of $\U(n)$ on $\mathfrak M_{k,n}$. Let
$\LL_k$ be the holomorphic line bundle on $\mathfrak M_{k,n}$
associated with the character
   \begin{equation}
   \label{eqn:9.8}
   \gamma_k: GL(k, \CC) \to \CC^*, \quad A \mapsto \det(A).
   \end{equation}
Then the bundle $\pi^*_k \LL_k$ is the holomorphic line bundle on
$M$ associated with $\gamma_k$ and
   \begin{equation}
   \label{eqn:9.9}
   \LL := \bigotimes_{k=1}^{n-1} (\pi_k^* \LL_k)^{m_k}
   \end{equation}
is the holomorphic line bundle associated with the character
$\gamma$. In particular if $s_k$ is a $GL(k, \CC)$-invariant
holomorphic section of $\LL_k$, then
   \begin{equation}
   \label{9.10}
   (\pi_1^* s_1)^{m_1} \cdots (\pi_{n-1}^* s_{n-1})^{m_{n-1}}
   \end{equation}
is a $G_\CC$-invariant holomorphic section of $\LL$, and all
$G_\CC$-invariant holomorphic sections of $\LL$ are linear
combinations of these sections. Since the representation of $GL(n,
\CC)$ on the space $\Gamma_{hol}(\LL_k)$ is its $k$-th elementary
representation we conclude
\begin{theorem}
\label{thm:9.3}
 The representation of $GL(n, \CC)$ on the space
$\Gamma_{hol}(\LL)$ is the irreducible representation with highest
weight $\sum_{i=1}^{n-1} m_i \alpha_i$, where $\alpha_1, \cdots,
\alpha_{n-1}$ are the simple roots of $GL(n, \CC)$.
\end{theorem}

For the canonical trivializing section of $\LL$ its Hermitian
inner product with itself is
   \begin{equation*}
   \prod_{i=1}^{n-1} \det(Z_i Z_{i+1}\cdots Z_{n-1} Z_{n-1}^*
   \cdots Z_i^*)^{-m_i}
   \end{equation*}
and hence the potential function for the $\LL$-twisted K\"ahler
structure on $M$ is
   \begin{equation}
   \label{eqn:9.11}
   \rho_\LL = \sum_{i=1}^{n-1} \tr Z_i Z_i^* - m_i \log
   \det (Z_i \cdots Z_{n-1}Z_{n-1}^* \cdots Z_i^*)
   \end{equation}
and the corresponding $\LL$-twisted moment map is
   \begin{equation}
   \label{eqn:9.12}
   \Phi_\LL(Z_1, \cdots, Z_{n-1}) = (Z_1 Z_1^* - m_1I_1,
   \cdots, Z_{n-1}Z_{n-1}^* - m_{n-1}I_{n-1}).
   \end{equation}
   %

%##############################
%        Section 10           %
%##############################

\section{Stability theory for coadjoint orbits}
\label{sec:10}

\subsection{The stability function on the Grassmannians $Gr(k,\CC^n)$}
\label{sec:10.1}
To compute this stability function, we first look for the
$G$-invariant sections of the twisted line bundle. For any index
set
   \begin{equation*}
   J = \{j_1, \cdots, j_k\} \subset \{1, 2, \cdots, n\}
   \end{equation*}
denote by $Z_J=Z_{j_1, \cdots, j_k}$ the $k \times k$ sub-matrix
consisting of the $j_1, \cdots, j_k$ columns of $Z$.
\begin{lemma}
\label{lemma:10.1}
The functions $s_J(Z)=\det(Z_J)$ are
$G$-invariant sections of the trivial line bundle on
$\mathfrak{M}_{k,n}$ for the twisted $G$-action.
\end{lemma}
\begin{proof}
Let $H$ be any $n \times n$ Hermitian matrix, and $v_H$ the
generator of the one-parameter subgroup generated by $H$. Then by
Kostant's identity (\ref{eqn:2.7}) one only needs to show
   \begin{equation*}
   \iota_{v_H} \partial \log\langle s_J, s_J\rangle = -\sqrt{-1}
   \tr\left((ZZ^*-I)H\right).
   \end{equation*}
This follows from direct computation:
   \begin{equation*}
   \aligned
   \iota_{v_H} \partial \log\langle s_J, s_J\rangle
   & = \iota_{v_H}\partial(-\tr ZZ^* +\log\det(Z_J \bar Z_J))\\
   & =  -\tr((\iota_{v_H}dZ)Z^*) + \iota_{v_H}\partial \tr
        \log(Z_J Z_J^*)\\
   & =  -\tr((\iota_{v_H}dZ)Z^*) + \tr((\iota_{v_H}dZ_J)Z_J^*
        (Z_J^*)^{-1} Z_J^{-1}) \\
   & = -\sqrt{-1}\tr(H(ZZ^*-I)),
   \endaligned
   \end{equation*}
completing the proof.
\end{proof}

Now we are ready to compute the stability function for the
Grassmannians. Without loss of generality, we suppose $\{j_1,
\cdots, j_k\} = \{1, \cdots, k\}$. For any rank $k$ matrix $Z \in
M_{st}$, let $B \in GL(k, \CC)$ be a nonsingular matrix with $BZ
\in \Phi^{-1}(I)$. Thus the stability function at point $Z$ is
   \begin{equation*}
   \aligned
   \psi(Z) &= \log\left(|\det(Z_{1, \cdots, k})|^2 e^{-\tr
   ZZ^*}\right)
   -\log\left(|\det((BZ)_{1, \cdots, k})|^2 e^{-\tr I}\right) \\
   & = k - \tr(ZZ^*) - \log{|\det B|^2}
   \endaligned
   \end{equation*}
Since $B^*B=(Z^*)^{-1}Z^{-1}$, we conclude
   \begin{equation}
   \label{10.1}
   \psi(Z) = k - \tr(ZZ^*) + \log \det(ZZ^*).
   \end{equation}

Similarly, if we do reduction at $mI$ instead of $I$, or
alternately,  use the moment map $\Phi-mI$, then the invariant
sections are given by $s_J(Z)=\det(Z_J)^m$, and the stability
function is
   \begin{equation*}
   \psi(Z) = km - \tr(ZZ^*) + m^2\log \det(ZZ^*).
   \end{equation*}

\subsection{The stability functions on $\U(n)$-coadjoint orbits}
\label{sec:10.2}
These stability functions are computed in more or less the same
way as above. By the same arguments as in the proof of lemma
\ref{lemma:10.1}, one can see that
   \begin{equation}
   \label{eqn:10.2}
   s(Z_1, \cdots, Z_{n-1})= \prod (\det(Z_i)_{1, \cdots,
   i})^{m_i-m_{i-1}}
   \end{equation}
is $G$-invariant for the moment map $\Phi-(m_1I_1, \cdots,
m_{n-1}I_{n-1})$.

Now suppose $(Z_1, \cdots, Z_{n-1}) \in M_{st}$, then there are
$B_i \in GL(i, \CC)$ such that
   \begin{equation}
   \label{eqn:10.3}
   B_1Z_1Z_1^*B_1^* = m_1 I_1
   \end{equation}
and
   \begin{equation}
   \label{eqn:10.4}
   B_iZ_iZ_i^*B_i^* = Z_{i-1}^* B_{i-1}^* B_{i-1}Z_{i-1} + m_i
   I_i, \qquad 2 \le i \le n-1.
   \end{equation}
>From (\ref{eqn:10.3}) we have
   \begin{equation*}
   \det(B_iB_1^*)=m_1\det(Z_1Z_1^*)^{-1},
   \end{equation*}
and from this and (\ref{eqn:10.4}) we conclude
   \begin{equation*}
   \aligned
   \det(B_iZ_iZ_i^*B_i^*)
   & = \det (m_i I_i + B_{i-1}Z_{i-1} Z_{i-1}^*B_{i-1}^*) \\
   & = \det ((m_i+m_{i-1})I_{i-1} + B_{i-2}Z_{i-2}
   Z_{i-2}^*B_{i-2}^*)
   \\
   & = m_1+\cdots+m_i.
   \endaligned
   \end{equation*}
So we get for all $i$,
   \begin{equation*}
   \det(B_iB_i^*) = (m_1+\cdots+m_i)\det(Z_iZ_i^*)^{-1}.
   \end{equation*}
Now it is easy to compute
   \begin{equation*}
   \aligned
   \psi(Z)
   & = \log \left(e^{-\sum \tr(Z_iZ_i^*)}
       \prod |\det(Z_i)_{1, \cdots, i}|^{2m_i-2m_{i-1}}
        \right)\\
   & \indent
       -\log \left( e^{-\sum i m_i}
      \prod |\det(B_iZ_i)_{1, \cdots, i}|^{2m_i-2m_{i-1}}
      \right) \\
   & = \sum im_i - \sum \tr(Z_iZ_i^*) - \sum (m_i-m_{i-1})
       \log |\det B_i|^2 \\
   & = \sum im_i - \sum \tr (Z_iZ_i^*) + \sum (m_i-m_{i-1})
       (m_1+\cdots+m_i) \log \det (Z_iZ_i^*).
   \endaligned
   \end{equation*}
\begin{remark}
Although we only carry out the computations for generic
$\U(n)$-coadjoint orbits, i.e., for the isospectral sets with
$\lambda_1 < \cdots < \lambda_n$, the same argument apply to all
$\U(n)$-coadjoint orbits. In fact, for the isospectral set with
$\lambda_1 < \cdots < \lambda_r$ whose multiplicities are $i_1,
\cdots, i_r$, we can take the upstairs space to be
   \begin{equation*}
   \mathfrak M_{i_1 \times (i_1+i_2)} \times \mathfrak
   M_{(i_1+i_2) \times (i_1+i_2+i_3)} \times \mathfrak
   M_{(n-i_r) \times n}
   \end{equation*}
and obtain results for these degenerate coadjoint orbits
completely analogous to those above.
\end{remark}

%##############################
%        Section 11           %
%##############################

\section{Stability functions on quiver varieties}
\label{sec:11}
It turns out that the results above can be generalized to a much
larger class of manifolds: quiver varieties. We will give a brief
account of this below.

\subsection{Quiver Varieties}
\label{sec:11.1}
Let's first recall some notations from quiver algebra theory. A
\emph{quiver} $Q$ is an oriented graph $(I, E)$, where $I=\{1, 2, \cdots,
n\}$ is the set of vertices, and $E \subset I \times I$ the set of
edges. A representation, $V$, of a quiver assigns a Hermitian
vector space $V_i$ to each vertex $i$ of the quiver and a linear
map $Z_{ij} \in \Hom(V_i, V_j)$ to each edge $(i,j) \in E$. The
dimension vector of the quiver representation $V$ is the vector
$l=(l_1, \cdots, l_n)$, where $l_i = \dim V_i$. Thus the space of
representations of $Q$ with underlying vector spaces $V$ fixed is
the complex space
   \begin{equation}
   \label{eqn:11.1}
   M = \Hom(V) := \bigoplus_{(i,j) \in E} \Hom(V_i, V_j).
   \end{equation}
We equip $M$ with its standard symplectic form and consider the
unitary group $U(V)=U(V_1) \times \cdots \times U(V_n)$ acting on
$M$ by
   \begin{equation}
   \label{eqn:11.2}
   (u_1, \cdots, u_n) \cdot (Z_{ij}) = (u_j Z_{ij} u_i^{-1}).
   \end{equation}
The isomorphism classes of representations of $Q$ of dimension $l$
is in bijection with the $GL(V)$-orbits on $\Hom(V)$.
Geometrically this quotient space can have bad singularities, and
to avoid this problem, one replaces this quotient by its GIT
quotient, or equivalently, the K\"ahler quotient of $\Hom(V)$ by
the $U(V)$-action. These quotients are what one calls \emph{quiver
varieties}.
\begin{proposition}
\label{prop:11.1}
The action (\ref{eqn:11.2}) is Hamiltonian with moment map $\mu:
\Hom(V) \to \fg^*$,
   \begin{equation}
   \label{eqn:11.3}
   \mu(Z_{ij}) = \left(\sum_{(j,1) \in E} Z_{j1}Z_{j1}^*-\sum_{(1,j)
   \in E}Z_{1j}^*Z_{1j},
   \cdots, \sum_{(j,n) \in E} Z_{jn}Z_{jn}^*-\sum_{(n,j) \in
   E}Z_{nj}^*Z_{nj}\right).
   \end{equation}
\end{proposition}
The proof involves the same computation as in lemma
\ref{lemma:9.1}, so we will omit it.

Notice that by (\ref{eqn:11.2}) the circle group
$\{(e^{i\theta}I_{l_1}, \cdots, e^{i\theta}I_{l_n})\}$ act
trivially on $M$, so we get an induced action of the quotient
group $G=U(V)/S^1$. The Lie algebra of $G$ is given by
   \begin{equation*}
   \{(H_1, \cdots, H_n)\ |\ H_i \mbox{\ Hermitian\ }, \sum
   \tr{H_i}=0\}
   \end{equation*}
and this $G$-action also has $\mu$ as its moment map. Letting
$(\lambda_1, \cdots, \lambda_n) \in \RR^n$  with
   \begin{equation*}
   l_1 \lambda_1 + \cdots + l_n \lambda_n = 0,
   \end{equation*}
and supposing that the $G$-action is free on $\mu^{-1}(\lambda
I)$, the \emph{quiver variety} associated to $\lambda$ is by
definition the quotient
   \begin{equation*}
   R_\lambda(l) = \mu^{-1}(\lambda I)/G,
   \end{equation*}
where $\lambda I = (\lambda_1 I_{l_1}, \cdots, \lambda_n
I_{l_n})$.

We can also modify the definition of quiver varieties to get an
effective $U(V)$-action. Namely, we attach to $Q$ another
collection of Hermitian vector spaces (the ``frame"), $\tilde
V=(\tilde V_1, \cdots, \tilde V_n)$,  with dimension vector
$\tilde l = (\tilde l_1, \cdots, \tilde l_n)$, and redefine the
space $M$ to be
   \begin{equation*}
   \Hom(V, \tilde V) := \bigoplus_{(i,j) \in E} \Hom(V_i, V_j) \oplus
                 \bigoplus_{i \in I} \Hom(V_i, \tilde V_i).
   \end{equation*}
The group $U(V)$ acts on $\Hom(V,\tilde V)$ by
   \begin{equation*}
   (u_1, \cdots, u_n) \cdot (Z_{ij}, Y_i) =
   (u_j Z_{ij} u_i^{-1}, Y_i u_i^{-1}).
   \end{equation*}
As above the $U(V)$-action is Hamiltonian, and
the $k^{th}$ component of its moment map is
   \begin{equation*}
   \left(\mu(Z_{ij}, Y_i)\right)_k = \sum_{(j,k) \in E}
   Z_{jk}Z_{jk}^*-\sum_{(k,j) \in E}Z_{kj}^*Z_{kj} - Y_k^* Y_k.
   \end{equation*}
Now the center $S^1$ acts nontrivially on $\Hom(V,\tilde V)$
providing that the ``frames" $\tilde V_i$ are not all zero, and we
define the \emph{framed quiver variety} $R_\lambda(l, \tilde l)$
to be the K\"ahler quotient of $\Hom(V, \tilde V)$ by the
$U(V)$-action above at the level $\lambda=(\lambda_1 I_{l_1},
\cdots, \lambda_n I_{l_n})$. As examples, the Grassmannian and the
coadjoint orbit of $\mathcal U(n)$ that we considered in the
previous section are just the framed quiver varieties whose
underlying quivers are depicted below:

\setlength{\unitlength}{1cm}
\begin{picture}(10,2)
\put(1,0.5){$\CC^n$} \put(1,1.5){$\CC^k$}
\put(1.1,1.47){\vector(0,-1){0.7}} \put(3,1.5){$\CC^1$}
\put(4.2,1.5){$\CC^2$} \put(6.7,1.5){$\CC^{n-2}$}
\put(4.6,1.62){\vector(1,0){0.7}} \put(8.3,1.5){$\CC^{n-1}$}
\put(8.3,0.5){$\CC^n$} \put(3.45,1.62){\vector(1,0){0.75}}
\put(5.4, 1.5){$\cdots$}
\put(5.85,1.62){\vector(1,0){0.7}}\put(7.5,1.62){\vector(1,0){0.7}}
\put(8.4,1.47){\vector(0,-1){0.7}}
\end{picture}

\subsection{Stability functions}
\label{sec:11.2}
As in \S 10 we equip $M$ with the trivial line bundle and, for
actions of $\mathcal U(V)$ associated with characters $\prod (\det
A_i)^{\lambda_i}$, describe the invariant sections.
\begin{proposition}
\label{prop:11.2}
For fixed $\lambda \in \ZZ^n$, the sections
   \begin{equation}
   \label{eqn:11.4}
   s(Z_{ij}) = \prod_{(i,j) \in E}\det((Z_{ij})_J)^{\nu_{ij}}
   \end{equation}
are invariant sections with respect to the moment map $\mu-
\lambda I$, where $\nu_{ij}$ are integers satisfying
   \begin{equation}
   \label{eqn:11.5}
   \sum_j \nu_{ji} - \sum_j \nu_{ij} = \lambda_i.
   \end{equation}
\end{proposition}
The proof is essentially the same proof as that of Lemma
\ref{lemma:10.1}.

>From now on we will require that the quiver,$Q$, be noncyclic,
otherwise there will be infinitely many $G$-invariant sections.
(Moreover, in the cyclic case the quiver variety is not compact.)
For a general quiver variety whose underlying quiver is noncyclic,
we can, in principle, compute the stability function, using the
$G$-invariant sections above, as we did for toric varieties in \S
7; but in practice the computation can be quite complicated.

However, in the special case that the quiver is a star quiver,
i.e., is of the following shape:

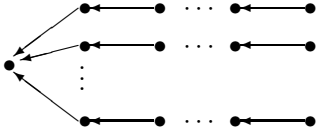
\setlength{\unitlength}{1cm}
\begin{picture}(6,2)
\put(1,1){$\bullet$} \put(2,0.25){$\bullet$}
\put(3,0.25){$\bullet$} \put(4,0.25){$\bullet$}
\put(5,0.25){$\bullet$} \put(2,0.75){$\vdots$}
\put(2,1.25){$\bullet$} \put(3,1.25){$\bullet$}
\put(4,1.25){$\bullet$} \put(5,1.25){$\bullet$}
\put(2,1.75){$\bullet$} \put(3,1.75){$\bullet$}
\put(4,1.75){$\bullet$} \put(5,1.75){$\bullet$}
\put(2,1.85){\vector(-4,-3){0.86}}
\put(2,1.35){\vector(-4,-1){0.77}}
\put(2,0.35){\vector(-4,3){0.86}}
\put(3,0.35){\vector(-1,0){0.85}}
\put(5,0.35){\vector(-1,0){0.85}}
\put(3,1.35){\vector(-1,0){0.85}}
\put(5,1.35){\vector(-1,0){0.85}}
\put(3,1.85){\vector(-1,0){0.85}}
\put(5,1.85){\vector(-1,0){0.85}} \put(3.4,0.25){$\cdots$}
\put(3.4,1.25){$\cdots$} \put(3.4,1.75){$\cdots$}
\end{picture}

\noindent one can write down the stability functions fairly
explicitly: on each ``arm", we just apply the same technique we
used for the coadjoint orbits of $\mathcal U(n)$.

As an example, we'll compute the stability function for
\emph{polygon space}. This is by definition a quiver variety whose
underlying quiver is the oriented graph

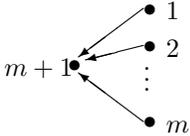
\setlength{\unitlength}{1cm}
\begin{picture}(4,2)
\put(1,1){$\bullet$}
\put(2,0.25){$\bullet$}
\put(2,0.75){$\vdots$}
\put(2,1.25){$\bullet$}
\put(2,1.75){$\bullet$}
\put(2,1.85){\vector(-4,-3){0.86}}
\put(2,1.35){\vector(-4,-1){0.77}}
\put(2,0.35){\vector(-4,3){0.86}}
\put(0.15,0.95){$m+1$}
\put(2.3,1.7){$1$}
\put(2.3,1.2){$2$}
\put(2.3,0.2){$m$}
\end{picture}

\noindent and for which the $V_i$'s satisfy $\dim V_i=1$ for $1
\le i \le m$ and $\dim V_{m+1}=2$. Thus
   \begin{equation}
   \label{eqn:11.6}
   \Hom(V)=\bigoplus \Hom(\CC, \CC^2) = (\CC^2)^m
   \end{equation}
and
   \begin{equation}
   \label{eqn:11.7}
   G = (S^1)^m \times U(2)/S^1 \simeq (S^1)^m \times SO(3).
   \end{equation}
The moment map for this data is
   \begin{equation}
   \label{eqn:11.8}
   (Z_1, \cdots, Z_m) \mapsto (-|Z_1|^2, \cdots, -|Z_m|^2,
   Z_1Z_1^*+\cdots+Z_mZ_m^*),
   \end{equation}
where $Z_i=(x_i, y_i) \in \CC^2$.

Now consider the quiver variety $\mu^{-1}(\lambda I)/G$, with
$\lambda=(\lambda_1, \cdots, \lambda_m, \lambda_{m+1})$ satisfying
   \begin{equation*}
   \lambda_1 + \cdots + \lambda_m+2\lambda_{m+1}=0
   \end{equation*}
and $\lambda_i<0$ for $1 \le i \le m$. Let's explain why this
variety is called ``polygon space". The $(S^1)^m$-action on
$(\CC^2)^m$ is the standard action, so reducing at level
$(\lambda_1, \cdots, \lambda_m)$ gives us a product of spheres
$S^2_{-\lambda_1} \times \cdots \times S^2_{-\lambda_m}$ of radii
$-\lambda_1, \cdots, -\lambda_m$. So we can think of an element of
$S^2_{-\lambda_1} \times \cdots \times S^2_{-\lambda_m}$ as a
polygon path in $\RR^3$ whose $i^{th}$ edge is a vector of length
$-\lambda_i$ in $S^2_{-\lambda_i}$. The $SO(3)$-action on this
product of spheres is the standard diagonal action, and the moment
map sums up the points, i.e. takes as its value the endpoint of
the polygon path. However, under the identification
(\ref{eqn:11.7}), the Lie algebra of $SO(3)$ gets identified with
$\mathcal H(2)/\{a I_2\}$. Thus the fact that the last entry of
the moment map (\ref{eqn:11.8}) equals $\lambda_{m+1}I_2$ implies
that this endpoint is the origin in the Lie algebra of $SO(3)$. In
other words, our polygon path is a polygon. So the quiver variety
$R_{\lambda}(1, \cdots, 1, 2)$ is just the space of all polygons
in $\RR^3$ whose sides are of length $-\lambda_1, \cdots,
-\lambda_m$, up to rotation.

Using the invariant section $s(Z)=\prod_{i=1}^m x_i^{-\lambda_i}$
to compute the stability function for this space we have
\begin{equation*}
\begin{aligned}
   \psi(Z) &= -\sum(|x_i|^2+|y_i|^2) + \sum (-\lambda_i)\log{|x_i|^2} \\
           &  \;\;\;\; \;\;\;+ \sum (-\lambda_i) -
             \sum (-\lambda_i)\log{\frac{-\lambda_i|x_i|^2}
             {|x_i|^2+|y_i|^2}} \\
            & = 2 \lambda_{m+1} - |Z|^2 + \sum \lambda_i
               \log \frac{-\lambda_i}{|Z_i|^2}.
   \end{aligned}
   \end{equation*}
Finally we point out that everything we have said above applies to
framed quiver varieties, in which case the $U(V)$-action is free
on $\Phi^{-1}(\lambda I)$. The coadjoint orbits of \S 10 are just
special cases of quiver varieties of this type.

%##############################
%        Appendix A           %
%##############################

\section*{appendix}

In this appendix we will give a proof of Boutet de Monvel-Guillemin's
theorem on the asymptotics of the density of states, (\ref{eqn:1.18}),
adapted to the Toeplitz operator setting.

We will begin with a very brief account on the definition of
Toeplitz operators. Let $W$ be a compact strictly pseudoconvex
domain with smooth boundary $\partial W$. One defines the space of
Hardy functions, $H^2$, to be the $L^2$-closure of the space of
$C^\infty$ functions on $\partial W$ which can be extended to
holomorphic functions on $W$. The orthogonal projection $\pi: L^2
\to H^2$ is called the Szeg\"o projector, and an operator $T:
C^\infty(\partial W) \to C^\infty(\partial W)$ is called a
Toeplitz operator if it can be written in the form
   \begin{equation*}
   T = \pi P \pi
   \end{equation*}
for some pseudodifferential operator $P$ on $\partial W$.

Now suppose $(\LL, \langle \cdot, \cdot \rangle)$ is a Hermitian
line bundle over a compact K\"ahler manifold $X$. Let
   \begin{equation*}
   D = \{(x, v) \in \LL^*\ |\ v \in \LL_x^*, \|v\|\le 1\}
   \end{equation*}
be the disc bundle in the dual bundle. As observed by Grauert, $D$
is a strictly pseudoconvex domain in $\LL$. The manifold we are
interested in is its boundary,
   \begin{equation*}
   M =\partial D = \{(x, v) \in \LL^*\ |\ v \in \LL_x^*,
   \|v\|=1\},
   \end{equation*}
the unit circle bundle in the dual bundle. Let $Q$ be the operator
   \begin{equation*}
   Q: H^2 \to H^2,\ Qf(x,v) = \sqrt{-1}\left.\frac{\partial}{\partial
   \theta} f(x, e^{i\theta}v)\right|_{\theta=0}.
   \end{equation*}
This is a first order elliptic operator in the Toeplitz sense and
is a Zoll operator (meaning that its spectrum only consists of
positive integers). Moreover, the $n^{th}$ eigenspace of $Q$
coincides with $\Gamma_{hol}(\LL^n, X)$. For any smooth function
$f \in C^\infty(X)$, let $M_f$ be the operator ``multiplication by
$f$". We may view $\Gamma_{hol}(\LL^n, X)$ as a subspace of $H^2$,
and denote by
   \begin{equation*}
   \pi_n: L^2(\LL^n, X) \to \Gamma_{hol}(\LL^n, X)
   \end{equation*}
the orthogonal projection.
\begin{theorema}
\label{theorem:A}
There is an asymptotic expansion
   \begin{equation*}
   \tr(\pi_n M_f \pi_n) \sim \sum_{k=d-1}^{-\infty} a_k(f) n^k, \quad n
   \to \infty,
   \end{equation*}
where $d=\dim{X}$.
\end{theorema}
\begin{proof}
By functorial properties of Toeplitz operators (c.f. \cite{BoG} \S
13),
   \begin{equation*}
   \tr(e^{itQ}M_f) \sim \sum a_k \chi_k(t)
   \end{equation*}
where
   \begin{equation*}
   \chi_k(t) = \sum_{n>0} n^k e^{int}.
   \end{equation*}
On the other hand,
   \begin{equation*}
   \tr(e^{itQ}M_f) = \sum e^{int}\tr{\pi_n M_f \pi_n}.
   \end{equation*}
By comparing the coefficient of $e^{int}$, we get the theorem.
\end{proof}

Finally we point out that the coefficients $a_k$ in the asymptotic
expansion above are given by the noncommutative residue trace on
the algebra of Toeplitz operators, \cite{Gui93}. In fact, for
$\Re(z) \gg 0$, theorem A gives
   \begin{equation*}
   \tr(Q^{-z}\pi_n M_f \pi_n) \sim \sum_{k=d-1}^{-\infty}
   a_k n^{k-z}.
   \end{equation*}
Summing over $n$,
   \begin{equation*}
   \tr(Q^{-z}M_f) \sim \sum_k a_k \zeta(z-k),
   \end{equation*}
where $\zeta$ is the classical zeta function, which implies
   \begin{equation*}
   a_{k-1} = \res_{z=k}(Q^{-z}M_f),
   \end{equation*}
the noncommutative residue.

%##############################
%        References           %
%##############################

\end{document}